\documentclass[11pt,reqno]{amsart}

\usepackage{amsmath, amsfonts, amsthm, amssymb, color, cite}

\newtheorem{Theorem}{Theorem}[section]
\newtheorem{Lemma}{Lemma}[section]
\newtheorem{Proposition}{Proposition}[section]

\theoremstyle{definition}
\newtheorem{Definition}{Definition}[section]

\theoremstyle{remark}
\newtheorem{Remark}{Remark}[section]

\numberwithin{equation}{section}
\allowdisplaybreaks

\def\va{\varphi}

\renewcommand{\i}{{\mathcal I}}
\renewcommand{\j}{{\mathcal J}}
\newcommand{\R}{{\mathbb R}}

\newcommand{\tr}{{\rm tr}}

\newcommand{\cd}{\cdot}
\newcommand{\na}{\nabla}
\newcommand{\dl}{\delta}
\def\f{\frac}
\renewcommand{\O}{\Omega}

\def\D{\Delta }

\def\hf1{^\f{1}{1-\xi^2}}
\def\l{\lambda}

\def\be{\begin{equation}}
\def\en{\end{equation}}
\def\bs{\begin{split}}
\def\es{\end{split}}
\def\ba{\begin{align}}
\def\ea{\end{align}}

\renewcommand{\a}{\alpha}

\renewcommand{\b}{\beta}
\newcommand{\del}{\partial}
\newcommand{\supp}{{\rm supp}}

\newcommand{\pa}{{\mathcal P}}
\newcommand{\la}{\lambda}
\newcommand{\s}{S_0}
\newcommand{\ve}{\varepsilon}

\author[G.-Q. Chen]{Gui-Qiang Chen}
\address{Mathematical Institute, University of Oxford, Oxford OX2 6GG, UK.}
\email{chengq@maths.ox.ac.uk}

\author[A. Majumdar]{Apala Majumdar}
\address{Department of Mathematical Sciences, University of Bath, Bath,  BA2 7AY,
UK.}
\email{a.majumdar@bath.ac.uk}

\author[D. Wang]{Dehua Wang}
\address{Department of Mathematics, University of Pittsburgh,
                           Pittsburgh, PA 15260, USA.}
\email{dwang@math.pitt.edu}

\author[R. Zhang]{Rongfang Zhang}
\address{Department of Mathematics, University of Pittsburgh,
                           Pittsburgh, PA 15260, USA.}
\email{roz14@pitt.edu}

\title[Active Liquid Crystals]
{Global Existence and Regularity of Solutions for Active Liquid Crystals}


\date{\today}

\begin{document}
\begin{abstract}
We study the hydrodynamics of active liquid crystals in the Beris-Edwards hydrodynamic framework
with the Landau-de Gennes $Q$-tensor order parameter to describe liquid crystalline ordering.
The existence of global weak solutions in two and three spatial dimensions is established.
In the two-dimensional case, by the Littlewood-Paley decomposition, the higher regularity of the weak solutions
and the weak-strong  uniqueness are also obtained.
\end{abstract}

\maketitle
\section{Introduction}

Liquid crystals are classical examples of mesophases that are intermediate between solids and liquids ({\it cf}. \cite{dg}).
They often combine physical properties of both liquids and solids, and
in general liquid crystals can be divided into thermotropic, lyotropic, and metallotropic phases,
according to their different optical properties.
Nematic liquid crystals are one of the most common liquid crystalline phases; nematics are complex liquids
with a certain degree of long-range orientational order. That is, the constituent molecules are typically rod-like or
elongated, and these elongated molecules flow about freely as in a conventional liquid but, whilst flowing,
they tend to align along certain distinguished directions ({\it cf}. \cite{dg,virga}).

There are several competing mathematical theories for nematic liquid crystals in the literature,
such as the Doi-Onsager theory proposed by Doi \cite{D-E-1986} in 1986 and Onsager\cite{O-1949} in 1949,
the Oseen-Frank theory proposed by Oseen \cite{O-1933} in 1933 and Frank \cite{F-1958} in 1958,
the Ericksen-Leslie theory proposed by Ericksen \cite{E-1961} in 1961 and Leslie \cite{L-1968} in 1968,
and the Landau-de Gennes theory proposed by Gennes \cite{G-1995} in 1995.
The first one is a molecular kinetic theory, and the remaining three are continuum macroscopic theories.
These theories can be derived or related to each other, under some assumptions.
For instance, Kuzzu-Doi \cite{K-D-1983} and E-Zhang \cite{E-Z-2006} formally derived
the Ericksen-Leslie equation from the Doi-Onsager equations by taking small Deborah number limit.
Wang-Zhang-Zhang \cite{W-Z-Z} justified this formal derivation before the first singular time of
the Ericksen-Leslie equations.
Wang-Zhang-Zhang \cite{W-Z-Z-1} presented a rigorous derivation of the Ericksen-Leslie equations
from the Beris-Edwards model
in the Landau-de Gennes framework. Ball-Majumdar \cite{B-M-2010}  and Ball-Zarnescu \cite{B-Z-2011} studied
the differences and the overlap between the Oseen-Frank theory and the Landau-de Gennes theory.
See \cite{L-L-2001, M-2010, M-Z-2010} for further discussions.

Active hydrodynamics describe fluids with active constituent particles that have collective motion and
are constantly maintained out of equilibrium by internal energy sources, rather than by the external forces applied to the system.
In particular, when the particles have elongated shapes, usually the collective motion induces the particles
to demonstrate orientational ordering at high concentration.
Thus, there are natural analogies with nematic liquid crystals. 
Active hydrodynamics have wide applications and have attracted much attention in recent decades.
For example, many biophysical systems are classified as active nematics, including microtubule bundles \cite{S-C-D-H-D-2012},
cytoskeletal filaments \cite{K-F-K-L-2008}, actin filaments \cite{C-D-2007},
dense suspensions of microswimmers \cite{W-D-H-D-G-L-Y-2012}, bacteria \cite{D-T-R-B-2010},
catalytic motors \cite{P-K-O-2004},
and even nonliving analogues such as monolayers of vibrated granular rods \cite{M-J-R-L-P-R-A-2013}.
For more information and discussions,
see  \cite{B-T-Y-2014, R-Y-2013, G-M-C-H-2011, G-M-C-H-2012, P-K-1992, D-E-1986} and the references therein.
Active nematic systems are distinguished from their well-studied passive counterparts since the constituent particles are active;
that is, it is the energy  consumed and dissipated by the active particles  that drives the system out of equilibrium,
rather than the external force applied at the boundary of the system, like a shear flow.
Consequently, active dynamics are truly striking, and many novel effects have been observed in active systems,
like the occurrence of giant density fluctuations \cite{R-S-T-2003, M-R-2006, N-R-M-2007},
the spontaneous laminar flow \cite{V-J-P-2005, M-O-C-Y-2007, G-M-L-2008},
unconventional rheological properties \cite{S-A-2009, G-L-M-2010, F-M-C-2011},
low Reynolds number turbulence \cite{W-D-H-D-G-L-Y-2012, G-M-C-H-2012}, and very different spatial and temporal patterns
compared to passive systems \cite{M-R-2006, C-G-M-2006, S-S-2008, G-P-B-C-2010, M-B-M-2010} arising
from the interaction of the orientational order and the flow.

In this paper, we use the Landau-de Gennes $Q$-tensor description that is one of the most comprehensive
descriptions, which describes the nematic state by a symmetric traceless $3\times 3$ matrix, the $Q$-tensor order parameter
with five independent degrees of freedom if the spatial dimension is three.
A nematic phase is said to be (i) isotropic if $Q=0$, (ii) uniaxial if $Q$ has a pair of degenerate non-zero eigenvalues,
and (iii) biaxial if $Q$ has three distinct eigenvalues.
In particular, a uniaxial phase has a single distinguished direction of nematic alignment,
and a biaxial phase has a primary and secondary direction of preferred alignment.
We remark that two-dimensional Q-tensors have been used to successfully model severely confined three-dimensional  nematic systems that are effectively invariant in the third dimension.

In particular, we consider the following hydrodynamic equations that model spatio-temporal
pattern formation in incompressible active nematic systems (see \cite{G-B-M-M-2013, G-M-C-H-2011}):
\be \label{Qtensor-1}
\begin{cases}
\partial_{t}Q+(u\cdot \nabla )Q+Q\Omega-\Omega Q-\lambda |Q|D=\Gamma H,\\
(\del_{t}+u_{\b}\del_{\b})u_{\a}+\del_{\a} P-\mu \D u_{\a} =\del_{\b} \tau_{\a \b}+\del_{\b}\sigma_{\a \b},\\
\nabla \cdot u=0,
\end{cases}
\en
where $u\in\R^d$, $d=2$ or $3$, is the flow velocity; $P$ is the pressure; $Q$ is the nematic tensor order parameter
that is a traceless and symmetric $d\times d$ matrix; $\mu>0$ denotes the viscosity coefficient;
$\Gamma^{-1}>0$ is the rotational viscosity; $\la \in \R$ stands for the nematic alignment parameter;
$D=\frac{1}{2}(\nabla u+\nabla u^{\top})$ and $\Omega=\frac{1}{2}(\nabla u-\nabla u^{\top})$ are the symmetric
and antisymmetric part of the strain tensor with $(\na u)_{\a\b}=\del_{\b}u_{\a}$.
Hereafter, we use the Einstein summation convention, {\it i.e.},  the repeated indices are summed over,
and $\alpha, \beta=1,2,\dots, d$. The time variable is $t\ge 0$, the space variable is $x=(x_1, \dots, x_d)$,
and $\partial_\beta=\frac{\partial}{\partial x_\beta}$.
Moreover, the molecular tensor
\begin{equation*}
\begin{split}
H=K\D Q-\frac{k}{2}(c-c_{*})Q+b\big(Q^2 -\frac{\tr(Q^2)}{d} \mathrm{I}_{d}\big)-cQ\,\tr(Q^2)
\end{split}
\end{equation*}
describes the relaxational dynamics of the nematic phase and can be obtained from
the Landau-de Gennes free energy, {\it i.e.}, $H_{\a \b}=-\frac{\dl F}{\dl Q_{\a \b}}$,
where
\begin{equation*}
F=\int \Big( \frac{k}{4}(c-c^{*})\tr(Q^2) -\frac{b}{3}\tr(Q^3) +\frac{c}{4} |\tr(Q^{2})|^{2}+\frac{K}{2}|\na Q|^{2}\Big)dA,
\end{equation*}
with $K$ the elastic constant for the one-constant elastic energy density,
$c$ the concentration of active units
and $c^*$ the critical concentration for the isotropic-nematic transition,
and $k>0$ and $b\in\R$ are  material-dependent constants.
We note that the analysis of this paper holds for all real $b$, 
but $b$ is usually taken to be positive in the literature.
In what follows, we set $K=k=1$ for simplicity of notation.
The stress tensor $\sigma_{\a \b}$ is
\begin{equation*}
\sigma_{\a \b}=\sigma_{\a \b}^{r}+\sigma_{\a \b}^{a},
\end{equation*}
with
$$
\sigma_{\a \b}^{r}=-\la |Q|H_{\a \b}+Q_{\a \dl}H_{\dl \b}-H_{\a \dl}Q_{\dl \b},\qquad
\sigma_{\a \b}^{a}=\sigma_{*} c^{2}Q_{\a \b},
$$
where $\sigma_{\a\b}^{r}$ is the elastic stress tensor due to the nematic elasticity,
and $\sigma_{\a \b}^{a}$ is the active contribution which describes contractile or extensile stresses exerted
by the active particles in the direction of the director field ($\sigma_{*}>0$ for the contractile case,
and $\sigma_{*}<0$ for the extensile case).
The symmetric additional stress tensor is denoted by
\begin{equation*}
\tau_{\a \b}:=-(\na Q\odot \na Q)_{\a \b}=-\del_{\b}Q_{\gamma \delta}\del_{\a}Q_{\gamma \delta}.
\end{equation*}

In the rest of this paper, we consider the case when $c>0$ is a constant. We set
\begin{equation*}
a=\frac{1}{2}(c-c_{*}), \quad \kappa=\sigma_{*} c^2.
\end{equation*}
Then system (\ref{Qtensor-1}) becomes
\be \label{Qtensor}
\begin{cases}
\partial_{t}Q+(u\cdot \nabla )Q+Q\Omega-\Omega Q-\lambda |Q|D=\Gamma H,\\
\del_{t}u+(u\cdot \na )u+\na P-\mu \D u =-\na \cdot (\na Q\odot \na Q)-\lambda \nabla \cdot (|Q|H)\\
\qquad \qquad \qquad \qquad \qquad \qquad \qquad+\nabla \cdot (Q\D Q-\D Q Q)+\kappa \na \cdot Q,\\
\nabla \cdot u=0,
\end{cases}
\en
with $$
H=\D Q-aQ+b\big(Q^2 -\frac{\tr(Q^2)}{d} \mathrm{I}_{d}\big)-cQ\,\tr(Q^2),
$$
the constants
$c>0, \, \Gamma >0, \, \mu>0, \, a, b, \lambda, \kappa \in \mathbb{R}$, and $(x, t)\in \R^{d}\times \R^{+}$.

Regarding the related mathematical contributions in the $Q$-tensor liquid crystal framework,
Paicu-Zarnescu \cite{P-Z-2011, P-Z-2012} proved the existence of global weak solutions to
the coupled incompressible Navier-Stokes and $Q$-tensor system for $d=2,3$,
as well as the existence of global regular solutions  with sufficiently
regular initial data for $d=2$.
Wilkinson \cite{W-2012} obtained the existence and regularity of weak solutions on the $d$--dimensional
torus over a certain singular potential.
In \cite{F-R-S-Z-2012}, Feireisl-Rocca-Schimperna-Zarnescu derived the global-in-time weak solutions
in the $Q$-tensor framework,
with arbitrary physically relevant initial data  in case of a singular
bulk potential proposed in Ball-Majumdar \cite{B-M-2010}.
Wang-Xu-Yu \cite{W-X-Y-2015} established the existence and long-time dynamics of globally
defined weak solutions for the coupled compressible Navier-Stokes and $Q$-tensor system.
See \cite{DeAnnaZar} and the references therein for more results and discussions.

We study the active system (\ref{Qtensor}) to establish the existence
of weak solutions in two and three spatial dimensions,
along with the existence of regular solutions and the uniqueness of weak-strong solutions in the two-dimensional case,
motivated by the work of Paicu-Zarnescu \cite{P-Z-2012, P-Z-2011} for the passive system.
Since we are dealing with the active system, we need to conquer some new difficulties.
Firstly, by using the general energy method, we obtain {\it a priori} estimates for the system (\ref{Qtensor}),
based on some crucial cancellations.
Those cancellations turn out to be very important in the proof of the existence of weak solutions in $\R^d, d=2,,3$,
higher regularity and the uniqueness of weak-strong solutions in $\R^2$.
However, due to the appearance of the active term $\kappa Q_{\a \b}$,
we can only obtain an energy inequality, instead of the perfect Lyapunov functional for the smooth solutions of the system.
Here we   mention that the symmetry and traceless properties of the $Q$-tensor play a key role in the validity
of the cancellations (see also Appendix \ref{appendix-a}).
Also the property of the $Q$-tensor (\ref{trace-Q-3}) is very important in order to derive the $H^1$-estimate
for the $Q$-tensor in Proposition \ref{energy-estimate},
since the bulk potential in the Landau-de Gennes energy
density (the terms independent of $\nabla Q$) is not always  positive.
Additionally, in order to obtain the weak solutions,
we need to add the extra terms $-\ve \del_{\a}Q_{\b \gamma}(u\cdot \na Q_{\b \gamma})|u\cdot \na Q|$
and $\ve \na \cdot (\na u|\na u|^{2})$ to the system
to control some non-vanishing terms in the energy estimates for the approximate system \eqref{approx-system-1},
due to the nonlinearity of the terms: $J_n (R_{\ve}u^n \nabla Q^{(n)})$ and
$J_n (R_{\ve}\Omega^n Q^{(n)}-Q^{(n)}R_{\ve}\Omega^n )$ (see \S 3 for the notations).
In \S \ref{a-priori-estimates}--\S \ref{weak-solutions},
the cancellations for these terms work very well.
However,  in \S \ref{2D-regularity}, when we seek the regular solutions in $\R^2$,
the cancellation for the terms, $\lambda |Q|D$ and $\lambda \nabla \cdot (|Q|H)$,
does not hold perfectly as before.
We use the Littlewood-Paley decomposition to reduce these  terms to two new terms
that can be controlled by the  previous cancellation idea (see Appendix B for the details).
We also need to pay more attention to the higher order terms of $Q$ in the elastic stress tensor of the system.

%

The rest of the paper is organized as follows:
In \S 2, we obtain the dissipation principle and {\it a priori} estimates.
In \S 3, we establish the existence of weak solutions in $\R^d, d=2,3$.
In \S 4, by using the Littlewood-Paley decomposition, we restrict ourselves to the two-dimensional case
and achieve the higher regularity of the corresponding weak solution.
In \S 5, we show the uniqueness of the weak and strong solution in $\R^2$,
with suitable initial data.
In Appendix A, we provide some important preliminary estimates that we use extensively in this paper.
In Appendix B, we provide the detailed estimates for inequality (\ref{high-freq-Q-u}).

\section{The Dissipation Principle and A Priori Estimates}\label{a-priori-estimates}

In this section, by using the energy method, we derive the dissipation principle
in Proposition \ref{energy-inequality} and the {\it a priori} estimates
in Proposition \ref{energy-estimate} for  system (\ref{Qtensor}).

For the sake of convenience,  we first introduce some notations.
We denote $H^{k}$, with $k\geq 1$ integer, as the Sobolev space
that consists of all functions $v$ in $L^{2}(\R^d)$
such that $D^{\nu}v$ is in $L^2(\R^d)$ for every multi-index $\nu=(\nu_{1}, \cdots, \nu_{d})$,
$0\leq |\nu|\leq k$, where ${D}^{\nu}:=\del_{1}^{\nu_{1}}\cdots \del_{d}^{\nu_{d}}$ is
the distributional derivative.
The space $H^{k}$ is equipped with norm $\|\cdot\|_{H^k}$ defined by
\begin{equation*}
\|v\|_{H^k}^2:=\sum_{0\leq |\nu|\leq k}\|{D}^{\nu}v\|_{L^2}^{2}.
\end{equation*}
The space $H^{-k}$, with $k\geq 1$ integer, is defined as the dual
spaces of $H^{k}_{0}$, equipped with the norm:
\begin{equation*}
\|v\|_{H^{-k}}:=\sup_{\va \in H^{k}_{0}, \|\va\|_{H^{k}_{0}}\leq 1} |(v, \va)|,
\end{equation*}
where $(\cdot, \cdot)$ stands for the inner product in $L^2$.
For example, if $a$ and $b$ are vector functions, then
\begin{equation*}
(a, b)=\int_{\R^d}a(x)\cdot b(x)\,dx,
\end{equation*}
and if $A$ and $B$ are matrices, then
\begin{equation*}
(A, B)=\int_{\R^{d}}A:B \,dx
\end{equation*}
with $A:B=\tr (AB)$.
We denote by $S_{0}^{d}\subset \mathbb{M}^{d\times d}$ the space of symmetric traceless $Q$-tensors in $d$-dimension,
that is,
\begin{equation*}
S_{0}^{d}:=\left\{Q\in \mathbb{M}^{d\times d}: \;  Q_{\a \b}=Q_{\b \a}, \ \tr(Q)=0, \ \a, \b=1, \cdots, d\right\}.
\end{equation*}
We define the norm of a matrix by using the Frobenius norm  denoted by
\begin{equation*}
|Q|:=\sqrt{\tr(Q^{2})}=\sqrt{Q_{\a \b}Q_{\a \b}}.
\end{equation*}
With respect to this norm, we can define the Sobolev spaces for the $Q$-tensors, for example,
\begin{equation*}
H^{1}(\R^d, S_{0}^{d})
:=\left\{Q: \R^{d}\rightarrow S_{0}^{d}\, : \  \int_{\R^d}\big(|Q(x)|^{2}+|\na Q(x)|^{2}\big) dx <\infty \right\}.
\end{equation*}
We also denote $|\na Q|^2:=\del_{\delta}Q_{\a \b}\del_{\delta}Q_{\a \b}$
and $|\D Q|^2:=\D Q_{\a \b}\D Q_{\a \b}$.

Let us denote the Landau-de Gennes free energy for the nematic liquid crystals ({\it cf.} \cite{G-1995})
by
\begin{align}
\mathbf{F}(Q)
:=\int_{\mathbb{R}^d }\left(\frac{1}{2}|\nabla Q|^{2}
 +\frac{a}{2}|Q|^{2}-\frac{b}{3}\tr (Q^3) +\frac{c}{4}|Q|^{4}\right)dx.
\end{align}
Moreover, by adding the kinetic energy to $\mathbf{F}(Q)$, we denote the energy of system (\ref{Qtensor}) by
\begin{align}
E(t):=\mathbf{F}(Q)+\frac{1}{2}\int_{\mathbb{R}^d }|u|^{2}dx.
\end{align}

\begin{Proposition} \label{energy-inequality}
Let $(Q, u)$ be a smooth solution of system \eqref{Qtensor} such that
\be
Q \in L^{\infty}(0, T; H^{1}(\mathbb{R}^d ))\cap L^{2}(0, T; H^{2}(\mathbb{R}^d ))
\en
and
\be
u\in L^{\infty}(0, T; L^{2}(\mathbb{R}^{d}))\cap L^{2}(0, T; H^{1}(\mathbb{R}^d ))
\en
for $d=2, 3$.
Then, for any given $T>0$,  we have
\be\label{2.5a}
\frac{d}{dt}E(t) +\frac{\mu}{2}\int_{\mathbb{R}^d }|\nabla u|^2 dx
+\Gamma \int_{\mathbb{R}^d }\tr(H^2)dx \leq C(\kappa, \mu)\int_{\mathbb{R}^d }|Q|^2 dx
\quad \mbox{for any $t\in (0,T)$}.
\en
\end{Proposition}

\begin{proof} We take the summation of the first equation in (\ref{Qtensor}) multiplied by $-H$
and the second equation in (\ref{Qtensor}) multiplied by $u$,  take the trace,
and then integrate by parts over $\R^d$ to find
\begin{align*}
&\frac{d}{dt}\int_{\R^{d}}\big(\frac{1}{2}|\nabla Q|^2 +\frac{a}{2}|Q|^2 -\frac{b}{3}\tr (Q^3)+\frac{c}{4}|Q|^4 +\frac{1}{2}|u|^2\big)dx
 +\mu\|\nabla u\|_{L^2}^2 +\Gamma \int_{\R^d}\tr(H^2)dx\\
&=(u\cdot \na Q, \D Q)-(u\cdot \na Q, aQ-b(Q^2 -\frac{\tr (Q^2)}{d}\mathrm{I}_d)+cQ|Q|^{2})-(\O Q-Q\O, \D Q)\\
&\quad +(\O Q-Q\O, aQ-b(Q^2 -\frac{\tr (Q^2)}{d}\mathrm{I}_d)
  +cQ|Q|^{2})-\la (|Q|D, H)\\
&\quad-(\na \cdot (\na Q\odot \na Q), u)+\la (|Q|H, \na u)+(\na \cdot (Q \D Q-\D QQ), u)-\kappa(Q, \na u)\\
&=\sum_{i=1}^{8}\mathcal{\i}_{i}-\kappa(Q, \nabla u)\\
& \leq \frac{\mu}{2}\|\nabla u\|_{L^2}^2 +C(\kappa, \mu)\|Q\|_{L^2}^2 ,
\end{align*}
where, in the last inequality, besides Cauchy's inequality, we have used that
$\i_{2}=0$ (since $\na \cdot u=0$), $\mathcal{\i}_3 +\mathcal{\i}_8 =0$ (by Lemma \ref{estimate-matrix}),
$\mathcal{\i}_1+\mathcal{\i}_6 =0$, $\mathcal{\i}_{4}=0$,
and $\mathcal{\i}_5 +\mathcal{\i}_7 =0$, as shown below:
\begin{align*}
\i_{1}+\i_{6} &=(u \cdot \na Q, \D Q)-(\na \cdot (\na Q \odot \na Q), u)\\
&=\int u_{\a}\del_{\a}Q_{\delta \gamma}\D Q_{\delta \gamma}dx
 -\int \del_{\a}\del_{\b}Q_{\delta \gamma}\del_{\b}Q_{\delta \gamma}u_{\a}dx
-\int \del_{\a}Q_{\delta \gamma}\del_{\b}\del_{\b} Q_{\delta \gamma}u_{\a}dx \\
&=-\int \del_{\a}\del_{\b}Q_{\delta \gamma}\del_{\b}Q_{\delta \gamma}u_{\a}dx\\
&=0,
\end{align*}
and, by the fact that $Q$ is symmetric and $\Omega$ is skew symmetric,
\begin{align*}
\i_{4}&=(\O Q-Q\O, aQ-b\big(Q^2 -\frac{\tr (Q^2)}{d}\mathrm{I}_d\big)+cQ|Q|^{2})\\
&=-(\O Q+Q\O, aQ-b\big(Q^2 -\frac{\tr (Q^2)}{d}\mathrm{I}_d\big)+cQ|Q|^{2})\\
&\quad+2(\O Q, aQ-b\big(Q^2 -\frac{\tr (Q^2)}{d}\mathrm{I}_d\big)+cQ|Q|^{2})\\
&=0,
\end{align*}
and
\begin{align*}
\mathcal{\i}_5 +\mathcal{\i}_7 &=\lambda (|Q|H, \nabla u)-\lambda (|Q|D, H)=\lambda (|Q|H, \nabla u)-\lambda (|Q|H, D)\\[2mm]
&=\lambda (|Q|H, \nabla u-D)
=\lambda (|Q|H, \Omega)=0.
\end{align*}
\end{proof}

\begin{Remark}
For the passive system considered in \cite{P-Z-2011}, a perfect Lyapunov functional is available.
However, for the active system as analyzed here, only an energy inequality \eqref{2.5a} is obtained above,
which is not a Lyapunov functional in general.
\end{Remark}

Based on Proposition \ref{energy-inequality} and Gronwall's inequality (Lemma \ref{Gronwall}), we have the following
{\it a priori} estimates.

\begin{Proposition}\label{energy-estimate}
Let $(Q, u)$ be a smooth solution of system \eqref{Qtensor} in $\R^d, d=2,3$,
with smooth initial data  $(\bar{Q}(x), \bar{u}(x))$.
If $(\bar{Q}, \bar{u})\in H^1\times L^2$, then, for any $t>0$,
\be\label{Q-H1}
\|Q(t, \cdot)\|_{H^1}\leq C_1 e^{C_2 t}\big(\|\bar{Q}\|^{2}_{H^1}+\|\bar{u}\|_{L^2}^2\big),
\en
and
\be\label{u-L2}
\frac{1}{2}\|u(t, \cdot)\|_{L^2}^2 +\frac{\mu}{4}\int_{0}^{t}\|\na u(s, \cdot)\|_{L^2}^2 ds
\leq C_{3}\big(\|\bar{Q}\|_{H^1}^2 +\|\bar{u}\|_{L^2}^2\big)e^{C_{2}t}+C_{4},
\en
where constants $C_i , 1\leq i\leq 4$, depend on $(a, b, c, \kappa, \mu, \la, \Gamma, \bar{Q}, \bar{u})$. 
\end{Proposition}

\begin{proof}
From the energy estimate in Proposition \ref{energy-inequality} and the symmetric property of $Q$, we have
\be\label{energy-estimate-equation}
\begin{split}
&\frac{d}{dt}E(t)+\frac{\mu}{2}\|\na u\|_{L^2}^{2}+\Gamma \|\D Q\|_{L^2}^{2}+a^2 \Gamma \|Q\|_{L^2}^{2}+c^2 \Gamma \|Q\|_{L^6}^{6}\\
&\quad+b^2 \Gamma \int_{\R^d}\tr ((Q^2 -\frac{\tr(Q^2)}{d}\mathrm{I}_d)^2) dx\\
&\,\,\leq\, C(\kappa, \mu) \|Q\|_{L^2}^2 +2a\Gamma (\D Q, Q)-2ac\Gamma \|Q\|_{L^4}^{4}-2b\Gamma (\D Q, Q^2)\\
&\quad \,+2bc\Gamma (Q\,\tr(Q^2), Q^2) +2ab\Gamma (Q, Q^2)+2c\Gamma (\D Q, Q\,\tr(Q^2))\\
&\,\, = C(\kappa, \mu) \|Q\|_{L^2}^2 -2a\Gamma \|\na Q\|_{L^2}^2 -2ac\Gamma \|Q\|_{L^4}^{4}+\sum_{i=1}^{4}\i_{i}.
\end{split}
\en
We now derive the estimates for  $\i_{i}$, $1\le i\le 4$. First, we have
\be\label{energy-i1}
\begin{split}
\i_{1}&=-2b\Gamma  (\D Q, Q^2)\leq \frac{\Gamma}{2}\|\D Q\|_{L^2}^2 +C(b^2 , \Gamma)\|Q\|_{L^4}^4. \\
\end{split}
\en
From (\ref{trace-Q-3}), we have the following estimates for $\i_2$ and $\i_3$ by choosing an appropriate $\ve>0$:
\be\label{energy-i2}
\begin{split}
\i_{2}&=2bc\Gamma (Q\,\tr(Q^2), Q^2)
=2bc\Gamma \int_{\R^d}\tr(Q^3) |Q|^2 dx\\
&\leq 2|b|c\Gamma \int_{\R^d}\big(\frac{\ve}{4} |Q|^4+\frac{1}{\ve}|Q|^2\big)|Q|^2 dx\\
&=\frac{c^2 \Gamma}{2} \|Q\|_{L^6}^{6}+C(b^2, \Gamma)\|Q\|_{L^4}^{4},\\
\end{split}
\en
and
\be\label{energy-i3}
\i_{3}=2ab\Gamma (Q, Q^2)=2ab \Gamma \int_{\R^d}\tr(Q^3) dx
\leq C(a, b, \Gamma) \big(\|Q\|_{L^2}^{2}+\|Q\|_{L^4}^{4}\big).
\en
Moreover, we observe that
\be\label{energy-i4}
\begin{split}
\i_{4}&=2c\Gamma (\D Q, Q\,\tr(Q^2))
=2c\Gamma \int_{\R^d}\del_{\gamma \gamma} Q_{\a \b}Q_{\a \b}\tr(Q^2) dx\\
&=-2c\Gamma \int_{\R^d}\del_{\gamma} Q_{\a \b}\del_{\gamma}Q_{\a \b}\tr(Q^2) dx
  -2c\Gamma \int_{\R^d}\del_{\gamma} Q_{\a \b}Q_{\a \b}\del_{\gamma}\tr(Q^2) dx\\
&=-2c\Gamma \int_{\R^d}|\na Q|^{2}|Q|^2 dx -c\Gamma \int_{\R^d}|\na \tr(Q^2)|^2 dx\leq 0.
\end{split}
\en
Combining (\ref{energy-estimate-equation}) with (\ref{energy-i1})$-$(\ref{energy-i4}), we have
\be\label{energy-e2}
\begin{split}
&\frac{d}{dt}E(t)+\frac{\mu}{2}\|\na u\|_{L^2}^{2}+\frac{\Gamma}{2} \|\D Q\|_{L^2}^{2}
 +\frac{c^2 \Gamma}{2} \|Q\|_{L^6}^{6}+b^2 \Gamma \int_{\R^d}\tr((Q^2 -\frac{\tr(Q^2)}{d}\mathrm{I}_d)^2) dx \\
&\leq C(a^2, b^2, c, \kappa, \mu, \Gamma)\big( \|\na Q\|_{L^2}^{2}+\|Q\|_{L^2}^{2} +\|Q\|_{L^4}^{4}\big).
\end{split}
\en
Here we should clarify that, since the potential terms, {\it i.e.}, the $Q$--terms without derivatives in $E(t)$
do not always sum to a positive quantity,
(\ref{energy-e2}) does not always yield the desired estimates for $Q$ directly.
However, we can deal with this issue as follows. \\

{\em Case I:  $a>0$ and sufficiently large.}
By the property of $Q$ in (\ref{trace-Q-3}), we have
\be
\begin{split}
\frac{a}{2}|Q|^2 -\frac{b}{3}\tr (Q^3) +\frac{c}{4}|Q|^{4}&\geq \frac{a}{2}|Q|^2
-\frac{b}{3}\left( \frac{\ve}{4}|\tr(Q^2)|^2+\frac{1}{\ve}\tr(Q^2)\right) +\frac{c}{4}|Q|^{4}\\
&=\left(\frac{a}{2}-\frac{b}{3\ve}\right)|Q|^2 +\left(\frac{c}{4}-\frac{b}{12}\ve\right)|Q|^4.
\end{split}
\en
When $\ve>0$ is sufficiently small and $a>0$ is sufficiently large,
both $\frac{a}{2}-\frac{b}{3\ve}$ and $\frac{c}{4}-\frac{b}{12}\ve$ can be positive.
As a result, we can obtain the $H^1$--estimates of $Q$ directly from (\ref{energy-e2})
by Gronwall's inequality (Lemma \ref{Gronwall}).

\smallskip
{\em Case II:  For all other $a$.}
In this case, the sum of the $Q$--terms without derivatives in $E(t)$ may be negative.
Thus, we have to deal with the $L^2$--estimates of $Q$ separately to obtain the $H^1$--estimates
for $Q$.
In fact, we multiply the first equation in (\ref{Qtensor}) by $Q$, take the trace,
and integrate over $\R^d$ by parts to obtain
\be\label{energy-Q-L2}
\begin{split}
\frac{1}{2}\frac{d}{dt}\|Q\|_{L^2}^2=&-\Gamma \|\na Q\|_{L^2}^2 -a\Gamma\|Q\|_{L^2}^2 -c\Gamma \|Q\|_{L^4}^4 +b\Gamma \int_{\R^d}\tr (Q^3)dx +\la (|Q|D, Q) \\
\leq& -\Gamma \|\na Q\|_{L^2}^2 -a\Gamma\|Q\|_{L^2}^2 -c\Gamma \|Q\|_{L^4}^4 +C(a^2, b^2, \Gamma)\big(\|Q\|_{L^2}^2 +\|Q\|_{L^4}^4\big)\\
&+\ve |\la| \|\na u\|_{L^2}^2 +C(\ve)|\la|\|Q\|_{L^4}^{4}\\[1mm]
\leq& -\Gamma \|\na Q\|_{L^2}^2 +\ve |\la|\|\na u\|_{L^2}^2+\bar{C}\big(\|Q\|_{L^2}^2 +\|Q\|_{L^4}^4\big),
\end{split}
\en
where $\ve>0$ will be decided later, and $\bar{C}=\bar{C}(a, b, c, \la, \Gamma, \ve)$.
Motivated by Case I, we notice that,
for any $Q\in S_{0}^{d}$, there exists a positive, sufficiently large constant $M=M(a, b, c)$ such that
\be\label{tr-Q-3}
0\leq\frac{M}{2}|Q|^2 +\frac{c}{8}|Q|^4 \leq (M+\frac{a}{2})|Q|^2 -\frac{b}{3}\tr(Q^3)+\frac{c}{4}|Q|^4 .
\en
Multiplying (\ref{energy-Q-L2}) by $2M$, adding it to (\ref{energy-e2}),
and letting $\ve=\frac{\mu}{8|\la|M}$, we have
\be
\begin{split}
&\frac{d}{dt}\big(E(t)+M\|Q\|_{L^2}^2\big)
 +\frac{\mu}{4}\|\na u\|_{L^2}^{2}+\frac{\Gamma}{2} \|\D Q\|_{L^2}^{2}+\frac{c^2 \Gamma}{2}\|Q\|_{L^6}^{6}\\
&\leq C\big( \|Q\|_{L^2}^{2} +\|\na Q\|_{L^2}^2 +\|Q\|_{L^4}^{4}\big),
\end{split}
\en
where $C=C(a, b, c, \kappa, \mu, \la, \Gamma, M)$.

Then the desired estimates (\ref{Q-H1})--(\ref{u-L2}) for $(Q, u)$ follow  from Gronwall's inequality (Lemma \ref{Gronwall}).
\end{proof}

\bigskip

\section{Weak Solutions}\label{weak-solutions}

In this section, we  prove the existence of weak solutions for system (\ref{Qtensor})
with suitable initial data for $d=2,3$.

\begin{Definition} $(Q, u)$ is called a weak solution of system (\ref{Qtensor}) with the initial data:
\be \label{initial-data}
Q(0, x)=\bar{Q}(x)\in H^1 (\R^{d}), \;\,\,\, u(0, x)=\bar{u}(x)\in L^2 (\R^d),\,\,\;\,
\,\,\,\nabla \cdot \bar{u}(x)=0 \,\,\,\, \mbox{in}\,\, \mathcal{D}^{'}(\R^d),
\en
if $(Q, u)$ satisfies the following:
\begin{enumerate}
\item[(i)] $Q\in L^{\infty}_{\rm loc}(\R_{+}; H^{1})\cap L^{2}_{\rm loc}(\R_{+}; H^{2})$ and
  $u\in L^{\infty}_{\rm loc}(\R_{+}; L^2)\cap L^{2}_{\rm loc}(\R_{+}; H^1)$;
\item[(ii)]  For every compactly supported $\varphi \in C^{\infty}([0, \infty)\times \R^d ; S_{0}^{d})$ and
  $\psi \in C^{\infty}([0, \infty)\times \R^{d}; \R^{d})$ with $\nabla \cdot \psi =0$,
\be \label{weak-solution-Q}
\begin{split}
&\int_{0}^{\infty}\int_{\R^d}\big(Q: \partial_{t}\varphi+\Gamma \D Q: \varphi+Q: (u\cdot\nabla_{x}\varphi) -(Q\Omega -\Omega Q-\lambda |Q|D): \varphi \big)dx dt\\
&\,\,=\Gamma \int_{0}^{\infty}\int_{\R^d}\Big(aQ-b\big(Q^2 -\frac{\tr(Q^2)}{d}\mathrm{I}_{d}\big) +cQ\,\tr(Q^2)\Big): \varphi\, dx dt
-\int_{\R^d}\bar{Q}(x): \varphi(0, x)\,dx,
\end{split}
\en
and
\be\label{weak-solution-u}
\begin{split}
&\int_{0}^{\infty}\int_{\R^d}\big(-u\cdot \partial_{t}\psi-u\cdot (u\cdot\na_{x}\psi)+\mu \nabla u: \nabla \psi^{\top}\big)dx dt
   -\int_{\R^d}\bar{u}(x)\cdot\psi(0, x)\,dx\\
&=\int_{0}^{\infty}\int_{\R^d}\Big(\na Q \odot \na Q+\lambda |Q|H -(Q\D Q-\D QQ) -\kappa Q\Big): \nabla \psi\, dx dt.
\end{split}
\en
\end{enumerate}
\end{Definition}

\smallskip
\begin{Theorem} \label{weak-solution-thm}
There exists a weak solution $(Q, u)$ of system \eqref{Qtensor} subject to
the initial conditions \eqref{initial-data}, for $d=2,3$, satisfying
\be
Q\in L^{\infty}_{\rm loc}(\R_{+}; H^{1})\cap L^{2}_{\rm loc}(\R_{+}; H^{2}),
\quad u\in L^{\infty}_{\rm loc}(\R_{+}; L^2)\cap L^{2}_{\rm loc}(\R_{+}; H^1).
\en
\end{Theorem}

\bigskip

Before proving Theorem \ref{weak-solution-thm}, we introduce some useful notations:
\begin{enumerate}
\item[(i)] $R_{\ve}$ is the convolution operator with kernel $\ve^{-d}\chi(\ve^{-1}\cdot)$,
 where $\chi \in C_{0}^{\infty}$ is a radial positive function such that
\begin{equation*}
\int_{\R^d}\chi(y)\mathrm{d}y=1.
\end{equation*}
\item[(ii)] The mollifying operator $J_n$, $n=1,2, \cdots$, is defined by
\begin{equation*}
\mathcal{F}(J_{n}f)(\xi):=1_{[2^{-n}, 2^{n}]}(|\xi |)\mathcal{F}({f})(\xi ),
\end{equation*}
where $\mathcal{F}$ is the Fourier transform.
\item[(iii)] $\pa$ is the Leray projector onto divergence-free vector fields, {\it i.e.},
\begin{equation*}
\pa :L^{2}(\Omega)\rightarrow H=\left\{ \mathbf{w}\in (L^{2}(\Omega))^{d}: \nabla \cdot \mathbf{w}=0 \right\},
\end{equation*}
which can be explicitly described in the Fourier domain by the tensor as
\begin{equation*}
\mathcal{F}(\pa (\mathbf{w}))(\xi)=\Big(\mathrm{I}_d-\frac{\xi \otimes \xi}{|\xi|^{2}}\Big)\mathcal{F}(\mathbf{w})(\xi).
\end{equation*}
\end{enumerate}

Next we prove Theorem \ref{weak-solution-thm} in the following three subsections.
In \S \ref{regularized-system},
we construct regularized approximate solutions $(Q^{(n)}_{ \ve}, u^{n}_{\ve})$
to the approximation system (\ref{approx-system}).
In \S \ref{n-goes-infty}, similarly to Proposition \ref{energy-estimate},
 we obtain some {\it a priori} bounds in (\ref{a-priori-bound}), which allow us to obtain
 the convergence result (\ref{convergence-result}) for $(Q^{(n)}_{ \ve}, u^{n}_{ \ve})$
 by the Aubin-Lions compactness lemma.
 Moreover, we can pass to the limit as $n$ goes to infinity to achieve the weak solution $(Q_{\ve}, u_{\ve})$
of the modified system (\ref{approx-system*}).
 In \S \ref{ve-goes-0}, by studying the $\ve\to 0$ limit of the modified system  (\ref{approx-system*}),
 we obtain the weak solution to system (\ref{Qtensor}).
 However, we cannot use the previous {\it a priori} bounds in (\ref{a-priori-bound}),
 since those bounds are not uniform with respect to $\ve$.
 Instead, we need to repeat a similar procedure
 in order to obtain the uniform bounds in (\ref{bound-approx-system-1})--(\ref{bound-approx-system-2}).

\subsection{Regularized Approximation System}\label{regularized-system}

Let us consider the following approximation system for the active hydrodynamic system (\ref{Qtensor}),
followed by the classical Friedrichs' scheme, for any fixed $\ve>0$ and $n>0$ (from now on,
solution $(Q^{(n)}_{\ve}, u^{n}_{\ve})$ is denoted by $(Q^{(n)}, u^{n})$ for simplicity
of notation when no confusion arises):
\be \label{approx-system}
\begin{cases}
\partial_{t}Q^{(n)} +J_n ((\pa J_n R_{\ve}u^n \cdot \nabla) J_n Q^{(n)})-J_n (\pa J_n (R_{\ve}\Omega^n) J_n Q^{(n)})\\
\quad \quad \quad \quad +J_n (J_n Q^{(n)}\pa J_n (R_{\ve}\Omega^n ))-\lambda J_n (|J_n Q^{(n)}|\pa J_n (R_{\ve}D^n ))=\Gamma J_{n}H^{(n)},\\[1.5mm]
\partial_{t}u^n +\pa J_n ((\pa J_n R_{\ve}u^n \cdot \nabla) \pa J_n u^n)-\mu \D \pa J_n u^n \\[0.5mm]
 \quad \quad \quad =-\ve \pa J_{n}R_{\ve}(\del_{\a}J_{n}Q^{(n)}_{\b \gamma}(R_{\ve}J_{n}u^{n}\cdot \na J_{n}Q^{(n)}_{\b \gamma})|R_{\ve}J_{n}u^{n}\cdot \na J_{n}Q^{(n)}|) \\
\quad \quad \quad \quad +\ve \pa \na \cdot J_{n}R_{\ve}(\na J_{n}R_{\ve}u^{n}|\na J_{n}R_{\ve}u^{n}|^{2})\\
\quad \quad \quad \quad -\pa \nabla \cdot J_n R_{\ve}(\nabla J_n Q^{(n)}\odot \nabla J_n Q^{(n)})-\lambda \pa \nabla \cdot J_n R_{\ve}(|J_n Q^{(n)}|J_{n}H^{(n)})\\
\quad \quad \quad \quad +\pa \nabla \cdot J_n R_{\ve}(J_n Q^{(n)}\D J_n Q^{(n)}-\D J_n Q^{(n)}J_{n}Q^{(n)})+\kappa \pa \nabla \cdot J_n R_{\ve} Q^{(n)},\\[1mm]
(Q^{(n)}, u^{n})|_{t=0}=(J_{n}R_{\ve}\bar{Q}, J_{n}R_{\ve}\bar{u}),
\end{cases}
\en
where
\begin{equation*}
H^{(n)}=\D Q^{(n)}-a Q^{(n)}+b\big((J_{n}Q^{(n)})^2 -\frac{\tr((J_{n}Q^{(n)})^2)}{d}\mathrm{I}_{d}\big)
-c J_{n}Q^{(n)}\tr((J_{n}Q^{(n)})^2).
\end{equation*}
This approximate system can be regarded as a system of ordinary differential equations in $L^2$.
By checking the conditions of the Cauchy-Lipschitz theorem \cite{T-2012},
we know that it admits a unique maximal solution
$(Q^{(n)}, u^n)\in C^{1}([0, T_n); L^2 (\R^d; \R^{d\times d})\times L^2 (\R^d, \R^d))$
on some time interval $[0, T_{n})$.
By simple calculation, $(\pa J_n)^2 =\pa J_n$ and $J_{n}^2 =J_n$, so that
pair $(J_n Q^{(n)}, \pa J_n u^n)$ is also a solution of (\ref{approx-system}).
By uniqueness, $(J_n Q^{(n)}, \pa J_n u^n)=(Q^{(n)}, u^n)$.
Therefore, $(Q^{(n)}, u^n)$ also satisfies the following system:
\be \label{approx-system-1}
\begin{cases}
\partial_{t}Q^{(n)} +J_n \big(R_{\ve}u^n \nabla Q^{(n)}\big)-J_n (R_{\ve}\Omega^n Q^{(n)}-Q^{(n)}R_{\ve}\Omega^n )-\lambda J_n (|Q^{(n)}|R_{\ve}D^n )
=\Gamma J_{n}\bar{H}^{(n)},\\[1mm]
\partial_{t}u^n +\pa J_n (R_{\ve}u^n \nabla u^n)-\mu \D u^n\\
\quad \quad \quad \quad \quad \quad=-\ve \pa J_{n}R_{\ve}\big(\del_{\a}Q^{(n)}_{\b \gamma}(R_{\ve}u^{n}\cdot \na Q^{(n)}_{\b \gamma})|R_{\ve}u^{n}\cdot \na Q^{(n)}|\big) \\
\quad \quad \quad \quad \quad \qquad  +\ve \pa \na \cdot J_{n}R_{\ve}(\na R_{\ve}u^{n}|\na R_{\ve}u^{n}|^{2})\\
\quad \quad \quad \quad \quad \qquad -\pa \nabla \cdot J_n R_{\ve} (\nabla Q^{(n)}\odot \nabla Q^{(n)})-\lambda \pa \nabla \cdot J_n R_{\ve}(|Q^{(n)}|J_{n}\bar{H}^{(n)})\\
\quad \quad \quad \quad \quad \qquad +\pa \nabla \cdot J_n R_{\ve}\big(Q^{(n)}\D Q^{(n)}-\D Q^{(n)}Q^{(n)}\big)+\kappa \pa \nabla \cdot  (J_n R_{\ve}Q^{(n)}),\\
(Q^{(n)}, u^{n})|_{t=0}=(J_{n}R_{\ve}\bar{Q}, J_{n}R_{\ve}\bar{u}),
\end{cases}
\en
where $(Q^{(n)}, u^n)\in C^1 ([0, T_n); \cap_{k=1}^\infty H^{k})$ and
\begin{equation*}
\bar{H}^{{n}}=\D Q^{(n)}-aQ^{(n)}+b\big((Q^{(n)})^2 -\frac{\tr((Q^{(n)})^2)}{d}\mathrm{I}_{d}\big)-cQ^{(n)}\tr((Q^{(n)})^2).
\end{equation*}

\begin{Remark}\label{simplicity}
It is easy to see that, if $Q^{(n)}$ is a solution to system \eqref{approx-system},
so is $(Q^{(n)})^{\top}$.
Hence, $Q^{(n)}=(Q^{(n)})^{\top}$ {\it a.e.}, in $[0, T_n]\times \R^{d}$,
by the uniqueness of the solution.
Moreover, from now on, we will work with the solution of the  system \eqref{approx-system-1} with
this symmetry property, instead of the solution to the system \eqref{approx-system}.
\end{Remark}

\subsection{Compactness and Convergence as $n\to \infty$ for System (\ref{approx-system-1})}\label{n-goes-infty}

First, we need to derive some {\it a priori} estimates for the system (\ref{approx-system-1}) in the following proposition.
\begin{Proposition}
The solution $(Q^{(n)}, u^{n})$ of the system \eqref{approx-system-1} satisfies the following estimates,
which are independent of $n$, for any $T<\infty$:
\be\label{a-priori-bound}
\begin{split}
&\sup_{n}\|R_{\ve}u^{n}\cdot \na Q^{(n)}\|_{L^{3}(0, T; L^{3})}
 +\sup_{n}\|\na R_{\ve}u^{n}\|_{L^{4}(0, T; L^{4})}\leq C,\\
&\sup_{n}\|Q^{(n)}\|_{L^{2}(0, T; H^{2})\cap L^{\infty}(0, T; H^{1}\cap L^{4})}+
\sup_{n}\|J_{n}(Q^{(n)}|Q^{(n)}|^{2})\|_{L^{2}(0, T; L^2)}\leq C,\\
&\sup_{n}\|u^{n}\|_{L^{\infty}(0, T; L^{2})\cap L^{2}(0, T; H^{1})}\leq C,
\end{split}
\en
provided the initial data $(\bar{Q}, \bar{u})\in H^{1}\times L^{2}$. Moreover, if $\bar{Q}\in S^{d}_{0}$, then $Q^{(n)}\in S^{d}_{0}$.
\end{Proposition}
\begin{proof} 

Similarly to the proof of Proposition \ref{energy-inequality}, we sum up the first equation in (\ref{approx-system-1})
multiplied by 
$-(\D Q^{(n)}-a Q^{(n)} +b(Q^{(n)})^{2}-cQ^{(n)}\tr ((Q^{(n)})^{2}))$ and the second equation multiplied by $u^{n}$, take the trace, and integrate by parts over $\R^d$ to obtain
\begin{align}\label{as1-sum}
&\frac{d}{dt}\int_{\R^d}\left(\frac{1}{2}|\nabla Q^{(n)}|^{2}+\frac{a}{2}|Q^{(n)}|^{2}
  -\frac{b}{3}\tr ((Q^{(n)})^{3})+\frac{c}{4}|Q^{(n)}|^{4}+\frac{1}{2}|u^n |^2 \right)dx\nonumber\\[1mm]
&\quad+\mu \|\nabla u^n\|_{L^2}^2 +\Gamma \|\D Q^{(n)}\|_{L^2}^2+a^{2}\Gamma \|Q^{(n)}\|_{L^2}^{2}+2ac\Gamma \|Q^{(n)}\|_{L^{4}}^{4}\nonumber\\
&\quad+b^{2}\Gamma \|J_{n}(Q^{(n)})^2 \|_{L^2}^{2}+c^{2}\Gamma \|J_{n}(Q^{(n)}|Q^{(n)}|^{2})\|_{L^{2}}^{2}+2a\Gamma \|\na Q^{(n)}\|_{L^{2}}^{2}\nonumber\\
&\quad +\ve \| R_{\ve}u^{n}\cdot \na Q^{(n)}\|_{L^{3}}^{3}+\ve \|\na R_{\ve}u^{n}\|_{L^{4}}^{4}\nonumber\\[1mm]
&=(R_{\ve}u^{n} \na Q^{(n)}, \D Q^{(n)})+(J_{n}(R_{\ve}u^{n}\na Q^{(n)}), bJ_{n}(Q^{(n)})^2 -c J_{n}(Q^{(n)}|Q^{(n)}|^{2}))\nonumber\\
&\quad-(R_{\ve}\O^{n}Q^{(n)}-Q^{(n)}R_{\ve}\O^{n}, \D Q^{(n)})+2c\Gamma (\D Q^{(n)}, Q^{(n)}|Q^{(n)}|^2 ) \\
&\quad+(J_{n}(R_{\ve}\O^{n}Q^{(n)}-Q^{(n)}R_{\ve}\O^{n}), -bJ_{n}(Q^{(n)})^2 +cJ_{n}(Q^{(n)}|Q^{(n)}|^{2}))\nonumber\\
&\quad-\la (|Q^{(n)}|R_{\ve}D^{n}, J_{n}(\D Q^{(n)}-a Q^{(n)}+b(Q^{(n)})^{2}-cQ^{(n)}|Q^{(n)}|^{2}))\nonumber\\
&\quad+2\Gamma ( -b\D Q^{(n)}+abQ^{(n)}+bc J_{n}(Q^{(n)}|Q^{(n)}|^{2}), J_{n}(Q^{(n)})^2 ) \nonumber\\
&\quad-(\na \cdot (\na Q^{(n)}\odot \na Q^{(n)}), R_{\ve}u^{n})+\la (|Q^{(n)}|J_{n}\bar{H}^{(n)}, \na R_{\ve}u^{n})\nonumber\\
&\quad+(\na \cdot (Q^{(n)} \D Q^{(n)}-\D Q^{(n)}Q^{(n)}), R_{\ve}u^{n}) -\kappa(Q^{(n)}, \na R_{\ve}u^{n})\nonumber\\
&\quad +\frac{b\Gamma}{d}(\tr((Q^{(n)})^{2})\mathrm{I}_{d}, J_{n}(\D Q^{(n)}-a Q^{(n)}+b(Q^{(n)})^{2}-cQ^{(n)}|Q^{(n)}|^{2})) \nonumber\\
&=\sum_{i=1}^{12}\mathcal{I}_{i}.\nonumber
\end{align}
By using cancellations analogous to \S 2, we have
$$
\mathcal{I}_{1}+\mathcal{I}_{8}=0, \quad
\mathcal{I}_{3}+\mathcal{I}_{10}=0, \quad
\mathcal{I}_{6}+\mathcal{I}_{9}=0,  \quad \mathcal{I}_{4}\leq 0.
$$
The remaining terms can be estimated as follows:
\begin{align*}
\i_{2}=&\,\big(J_{n}(R_{\ve}u^{n}\na Q^{(n)}), bJ_{n}(Q^{(n)})^2 -cJ_{n}(Q^{(n)}|Q^{(n)}|^{2})\big)\\
\leq &\, (\frac{2}{ \Gamma}+\frac{1}{4C(b^2)})\|R_{\ve}u^{n}\cdot \na Q^{(n)}\|_{L^{2}}^{2}+\frac{c^2 \Gamma}{8}\|J_{n}(Q^{(n)}|Q^{(n)}|^{2})\|_{L^{2}}^{2}+C(b^2)\|Q^{(n)}\|_{L^4}^4 \\
\leq &\, \frac{\ve}{2}\|R_{\ve}u^{n}\cdot \na Q^{(n)}\|_{L^3}^{3}+C(b^2, \ve, \Gamma)\|R_{\ve}u^{n}\cdot \na Q^{(n)}\|_{L^1} \\
&\, +\frac{c^2 \Gamma}{8}\|J_{n}(Q^{(n)}|Q^{(n)}|^{2})\|_{L^{2}}^{2}+C(b^2)\|Q^{(n)}\|_{L^4}^4\\
\leq & \, \frac{\ve}{2}\|R_{\ve}u^{n}\cdot \na Q^{(n)}\|_{L^3}^{3}+C(b^2, \ve, \Gamma)\left(\|u^{n}\|_{L^2}^2 +\|\na Q^{(n)}\|_{L^{2}}^{2}\right)\\
&\, +\frac{c^2 \Gamma}{8}\|J_{n}(Q^{(n)}|Q^{(n)}|^{2})\|_{L^{2}}^{2}+C(b^2 )\|Q^{(n)}\|_{L^4}^4,
\end{align*}
\begin{align*}
\i_{5}= & \, \big(J_{n}(R_{\ve}\O^{n}Q^{(n)}-Q^{(n)}R_{\ve}\O^{n}), -bJ_{n}(Q^{(n)})^2 +cJ_{n}(Q^{(n)}|Q^{(n)}|^{2})\big)\\
\leq &\, \big(\frac{4}{\Gamma}+\frac{1}{2C(b^2)}\big)\|R_{\ve}\O^{n}Q^{(n)}\|_{L^{2}}^{2}+\frac{\Gamma c^{2}}{8}\|J_{n}(Q^{(n)}|Q^{(n)}|^{2})\|_{L^{2}}^{2}+C(b^2 )\|Q^{(n)}\|_{L^4}^4\\
\leq &\, \frac{\ve}{2}\|\na R_{\ve}u^{n}\|_{L^{4}}^{4}+C(b^2, \Gamma, \ve)\|Q^{(n)}\|_{L^{4}}^{4}+\frac{\Gamma c^{2}}{8}\|J_{n}(Q^{(n)}|Q^{(n)}|^{2})\|_{L^{2}}^{2},
\end{align*}
\begin{align*}\label{2b}
\begin{split}
\i_{7}&=2\Gamma \big( -b\D Q^{(n)}+abQ^{(n)}+bc J_{n}(Q^{(n)}|Q^{(n)}|^{2}), J_{n}(Q^{(n)})^2 \big)\\
&\leq \frac{\Gamma}{4}\|\D Q^{(n)}\|_{L^2}^{2}
+\frac{c^2 \Gamma}{8}\|J_{n}(Q^{(n)}|Q^{(n)}|^{2})\|_{L^{2}}^{2}+C(a^2, b^2, \Gamma)\big(\|Q^{(n)}\|_{L^2}^2 +\|Q^{(n)}\|_{L^4}^4\big),
\end{split}
\end{align*}
\begin{align*}
\mathcal{I}_{11}& =-\kappa(Q^{(n)}, \na R_{\ve}u^{n})\leq \frac{\mu}{4}\|\na R_{\ve}u^{n}\|_{L^{2}}^{2}+C(\kappa^2, \mu)\|Q^{(n)}\|_{L^{2}}^{2}\\
&\leq \frac{\mu}{4}\|\na u^{n}\|_{L^{2}}^{2}+C(\kappa^2, \mu)\|Q^{(n)}\|_{L^{2}}^{2},\\
\i_{12}&=\frac{b\Gamma}{d}(\tr((Q^{(n)})^{2})\mathrm{I}_{d}, J_{n}(\D Q^{(n)}-a Q^{(n)}+b(Q^{(n)})^{2}-cQ^{(n)}|Q^{(n)}|^{2}))\\
&\leq \frac{\Gamma}{4}\|\D Q^{(n)}\|_{L^{2}}^{2} +\frac{c^{2}\Gamma}{8}\|J_{n}(Q^{(n)}|Q^{(n)}|^{2})\|_{L^{2}}^{2}+C\|Q^{(n)}\|_{L^{2}}^{2}+C\|Q^{(n)}\|_{L^{4}}^{4}.
\end{align*}
Substituting all the above estimates into (\ref{as1-sum}), we have
\be\label{energy-for-approx-1}
\begin{split}
&\frac{d}{dt}E^{n}(t)+\frac{3\mu}{4} \|\nabla u^n\|_{L^2}^2 +\frac{\Gamma}{2} \|\D Q^{(n)}\|_{L^2}^2 +\frac{c^{2}\Gamma}{2} \|J_{n}(Q^{(n)}|Q^{(n)}|^{2})\|_{L^{2}}^{2}\\
&\quad+\frac{\ve}{2} \|R_{\ve}u^{n}\cdot \na Q^{(n)}\|_{L^{3}}^{3}+\frac{\ve}{2} \|\na R_{\ve}u^{n}\|_{L^{4}}^{4}\\[1mm]
&\,\,\leq C\left(\|Q^{(n)}\|_{L^{2}}^{2}+\|Q^{(n)}\|_{L^{4}}^{4}+\|u^{n}\|_{L^{2}}^2 +\|\na Q^{(n)}\|_{L^2}^2 \right),
\end{split}
\en
where $C$ depends on $a, b, c, \kappa, \Gamma, \mu$, and $\ve$, and
\begin{equation*}
E^{n}(t)=\int_{\R^d}\left(\frac{1}{2}|\nabla Q^{(n)}|^{2}+\frac{a}{2}|Q^{(n)}|^{2}
  -\frac{b}{3}\tr ((Q^{(n)})^{3})+\frac{c}{4}|Q^{(n)}|^{4}+\frac{1}{2}|u^n |^2 \right)dx.
\end{equation*}

Again, by the same reasoning as in Proposition \ref{energy-estimate}, $E^{n}(t)$ may be negative.
We also need to estimate the $L^2$--norm of the $Q$-tensor separately in order to obtain the desired $H^1$--estimates for $Q$.
Multiplying the first equation in system (\ref{approx-system-1}) by $Q^{(n)}$,
taking the trace, integrating over $\R^d$ by parts,
multiplying the result by $2M$ ($M>0$ is sufficiently large),
and adding it to (\ref{energy-for-approx-1}), we have
\begin{align*}
&\frac{d}{dt}\left(E^{n}(t)+M\|Q^{(n)}\|_{L^2}^2\right)+\frac{\mu}{2} \|\nabla u^n\|_{L^2}^2 +\frac{\Gamma}{2} \|\D Q^{(n)}\|_{L^2}^2 \\
&\quad +\frac{c^{2}\Gamma}{2} \|J_{n}(Q^{(n)}|Q^{(n)}|^{2})\|_{L^{2}}^{2}+\frac{\ve}{2} \|R_{\ve}u^{n}\cdot \na Q^{(n)}\|_{L^{3}}^{3}+\frac{\ve}{2} \|\na R_{\ve}u^{n}\|_{L^{4}}^{4}\\
&\,\, \leq C\left(\|Q^{(n)}\|_{L^{2}}^{2}+\|Q^{(n)}\|_{L^{4}}^{4}+\|u^{n}\|_{L^{2}}^2 +\|\na Q^{(n)}\|_{L^2}^2 \right),
\end{align*}
where $C=C(a, b, c, \kappa, \la, \Gamma, \ve, \mu, M)$ is a constant, independent of $n$.

From the above estimate, along with Gronwall's inequality (Lemma \ref{Gronwall}), we can conclude the {\it a priori} bounds in \eqref{a-priori-bound},
which are independent of $n$, for any $T<\infty$.

In order to prove that $Q^{(n)}\in S^{d}_{0}$, besides the symmetry property mentioned in Remark \ref{simplicity},
it remains to show $\tr (Q^{(n)})=0$.
We take the trace on both sides of the first equation in system \eqref{approx-system-1} and
use   that $Q^{(n)}=(Q^{(n)})^{\top}$, $(\O^{n})^{\top}=-\O^{n}$, and $\tr (D^{n}) =\mathrm{div}(u^{n})=0$
to obtain the following initial value problem:
\begin{align*}
&\del_{t} \tr (Q^{(n)})+J_n \big(R_{\ve}u^n \cdot \nabla \tr (Q^{(n)})\big)=\Gamma J_{n}(\D \tr(Q^{(n)})-a\,\tr (Q^{(n)})-c\, \tr (Q^{(n)}) \tr((Q^{(n)})^{2})),\\
&\tr (Q^{(n)})|_{t=0}=J_{n}R_{\ve}\tr(\bar{Q})=0.
\end{align*}
Multiplying the above equation by $\tr(Q^{(n)})$, integrating by parts over $\R^{d}$, and using $J_{n}Q^{(n)}=Q^{(n)}$ and
the uniform bounds of $Q^{(n)}$ in \eqref{a-priori-bound}, we have
\begin{align*}
&\frac{d}{d t}\|\tr(Q^{(n)})\|_{L^{2}}^{2}+\Gamma \|\na \tr(Q^{(n)})\|_{L^{2}}^{2}=-a \Gamma \|\tr(Q^{(n)})\|_{L^{2}}^{2}-c\Gamma \int_{\R^{d}}|\tr(Q^{(n)})|^{2}|Q^{(n)}|^{2}d x\\
&\leq -a\Gamma \|\tr(Q^{(n)})\|_{L^{2}}^{2}+C\|Q^{(n)}\|_{L^{6}}^{2}\|\tr(Q^{(n)})\|_{L^{6}}\|\tr(Q^{(n)})\|_{L^{2}}\\
&\leq -a\Gamma \|\tr(Q^{(n)})\|_{L^{2}}^{2}+C\|Q^{(n)}\|_{H^{1}}^{2}\|\na \tr(Q^{(n)})\|_{L^{2}}^{\frac{d}{3}}\|\tr(Q^{(n)})\|_{L^{2}}^{2-\frac{d}{3}}\\
&\leq \frac{\Gamma}{2}\|\na \tr(Q^{(n)})\|_{L^{2}}^{2}+C\|\tr(Q^{(n)})\|_{L^{2}}^{2},
\end{align*}
thus,
$$
\frac{d}{d t}\|\tr(Q^{(n)})\|_{L^{2}}^{2}\le C\|\tr(Q^{(n)})\|_{L^{2}}^{2},
$$
where we have used the Sobolev imbedding, the Gagliardo-Nirenberg interpolation inequality in Lemma \ref{gn-inequality}£¬ and the Cauchy inequality.
Hence, we conclude that $\tr(Q^{(n)})=0$ by the initial condition.
\end{proof}

Then we can conclude from the uniform estimates in \eqref{a-priori-bound} that $T_{n}=\infty$.
In addition, by using system (\ref{approx-system-1}) and the above estimates,
we can compute the bounds for $\del_{t}(Q^{(n)}, u^{n})$
in some $L^{1}(0, T; H^{-N})$ for large enough $N$.
Then, by the classical Aubin-Lions compactness lemma (Lemma \ref{al-lemma}),
we conclude that, subject to a subsequence,
\be\label{convergence-result}
\begin{split}
&Q^{(n)}\rightharpoonup Q \quad \mbox{in}\; L^{2}(0, T; H^{2}),
 \qquad Q^{(n)}\rightarrow Q \quad \mbox{in} \; L^{2}(0, T; H_{\rm loc}^{2-\delta}) \,\,\,\mbox{for any} \, \delta \in(0, 2+N),\\
&Q^{(n)}(t)\rightharpoonup Q(t) \qquad \mbox{in} \; H^{1} \,\,\, \mbox{for any} \,  t\in \R_{+},\\
&Q^{(n)}\rightharpoonup Q\; \mbox{in}\; L^{p}(0, T; H^1),
 \; Q^{(n)}\rightarrow Q \; \mbox{in}\; L^{p}(0, T; H^{1-\delta}_{\rm loc})\;\; \mbox{for any}\, \delta \in (0, 1+N),\; p\in [2, \infty],\\
&u^{n}\rightharpoonup u \quad \mbox{in}\; L^{2}(0, T; H^{1}),
  \qquad u^{n}\rightarrow u \quad \mbox{in}\; L^{2}(0, T; H_{\rm loc}^{1-\delta}) \,\,\, \mbox{for any} \,\delta \in (0, 1+N),\\
&u^{n}(t)\rightharpoonup u(t)\quad \mbox{in} \; L^{2}\quad \mbox{for any \,$t\in \R_{+}$}.
\end{split}
\en
As a result, we can pass to the limit as $n$ goes to infinity to obtain a weak solution $(Q_{\ve}, u_{\ve})$
 of  the following modified system (for simplicity, we denote $(Q_{\ve}, u_{\ve})$ by $(Q, u)$ when no confusion arises):
\be \label{approx-system*}
\begin{cases}
\partial_{t}Q + R_{\ve} u \cdot \na Q+QR_{\ve}\O-R_{\ve}\O Q-\lambda |Q| R_{\ve}D=\Gamma H,\\
\partial_{t}u +\pa (R_{\ve} u\cdot \na u)-\mu \D u\\
  \quad = -\ve \pa R_{\ve}(\del_{\a}Q_{\b \gamma}(R_{\ve}u\cdot \na Q_{\b \gamma})|R_{\ve}u\cdot \na Q|)
   +\ve \pa \na \cdot R_{\ve}(\na R_{\ve}u|\na R_{\ve}u|^{2})\\
  \quad\,\,\,\,\, \, -\pa \nabla \cdot R_{\ve}(\na Q\odot \na Q)
   -\lambda \na \cdot \pa R_{\ve} (|Q|H)+\pa \na \cdot R_{\ve} (Q\D Q-\D QQ)+\kappa \pa \na \cdot R_{\ve} Q,\\
(Q, u)|_{t=0}=(R_{\ve}\bar{Q}, R_{\ve}\bar{u}),
\end{cases}
\en
such that
\be
Q_{\ve}\in L_{\rm loc}^{\infty}(\R_{+}; H^{1})\cap L_{\rm loc}^{2}(\R_{+}; H^{2}), \quad u_{\ve}\in L_{\rm loc}^{\infty}(\R_{+}; L^{2})\cap L_{\rm loc}^{2}(\R_{+}; H^{1}).
\en

\begin{Remark}\label{remark-1}
This modified system is obtained by mollifying the coefficients of the $Q$-tensor equation
and the forcing terms of the velocity equation, and by adding the extra terms
given by $-\ve \del_{\a}Q_{\b \gamma}(u\cdot \na Q_{\b \gamma})|u\cdot \na Q|$
and $\ve \na \cdot (\na u|\na u|^{2})$ to the velocity equation.
These two terms are needed  to estimate some {\it bad} terms which do not disappear
in the energy estimates.
Moreover, we would like to mention that the above procedure for obtaining the solution to system (\ref{approx-system*})
follows from the classical Friedrichs' scheme. We also  point out that the solutions to (\ref{approx-system*})
are smooth, because we can bootstrap the regularity improvement provided by the linear heat equation
to obtain the smooth regularity of $Q$, and bootstrap the regularity improvement
provided by the linear advection equation to obtain the smooth regularity of $u$.
\end{Remark}

\subsection{Compactness and Convergence as $\ve \to 0$ for System (\ref{approx-system*})}\label{ve-goes-0}
In this subsection, we show that, by passing to the   $\ve\to 0$ limit in system (\ref{approx-system*}),
we can obtain a weak solution of the  system (\ref{Qtensor}).
In order to do so, we need to achieve some uniform bounds for solution $(Q_{\ve}, u_{\ve})$.
Although, as a limit of $(Q^{(n)}_{\ve}, u^{n}_{\ve})$, $(Q_{\ve}, u_{\ve})$ still satisfies
the {\it a priori} bounds  in (\ref{a-priori-bound}),
we cannot apply these bounds in this step because they are not uniform with respect to $\ve$.
Therefore, we need to find  new {\it a priori} bounds for the system (\ref{approx-system*}).
For simplicity of notation, we denote $(Q_{\ve}, u_{\ve})$ by $(Q, u)$ when no confusion arises below.

Applying the same procedure as in \S \ref{a-priori-estimates}, that is, multiplying the first equation
in (\ref{approx-system*}) by $-H+2MQ$ (with $M$ a sufficiently large positive constant),
taking the trace, integrating by parts over $\R^d$, and adding this   to the second equation
in (\ref{approx-system*}) multiplied by $u$
and integrated by parts over $\R^d$, with the analogous cancellations
as  in \S \ref{a-priori-estimates}, we have
\be
\begin{split}
& \frac{d}{dt}\int_{\R^d }\left(\frac{1}{2}|\nabla Q|^2 +(\frac{a}{2}+M)|Q|^2 -\frac{b}{3}\tr (Q^3)
 +\frac{c}{4}|Q|^{4}+ \frac{1}{2}|u|^{2} \right)dx \\
&\quad +\frac{\mu}{2}\|\na u\|_{L^2}^{2} +\frac{\Gamma}{2} \|\D Q\|_{L^2}^{2}
 +\frac{c^2 \Gamma}{2} \|Q\|_{L^6}^{6}+\ve \|R_{\ve}u\cdot \na Q\|_{L^{3}}^{3}
  +\ve \|\na R_{\ve}u\|_{L^{4}}^{4}\\
&\,\,\,\leq\,  C(\|Q\|_{L^2}^{2} +\|\na Q\|_{L^2}^2 +\|Q\|_{L^4}^4 ),
\end{split}
\en
where $C=C(a, b, c, \kappa, \mu, \la, \Gamma, M)$ is independent of $\ve$.
Then, by Gronwall's inequality (Lemma \ref{Gronwall}), we have the following {\it a priori} bounds, independent of $\ve$, such that, for any $T<\infty$,
\be\label{bound-approx-system-1}
\sup_{\ve}\|Q_{\ve}\|_{L^{2}(0, T; H^{2})\cap L^{\infty}(0, T; H^{1}\cap L^{4})\cap L^{6}(0, T; L^6)}
+ \sup_{\ve}\|u_{\ve}\|_{L^{\infty}(0, T; L^{2})\cap L^{2}(0, T; H^{1})}\leq C.
\en
Moreover, for any $\ve>0$, we have
\be\label{bound-approx-system-2}
\ve\|R_{\ve}u\cdot \na Q\|^{3}_{L^{3}(0, T; L^{3})}
+\ve\|\na R_{\ve}u\|^{4}_{L^{4}(0, T; L^{4})}\leq C
\en
with constant $C$ independent of $\ve$.
In addition, since $(Q_{\ve}, u_{\ve})$ satisfies system (\ref{approx-system*}),
along with  (\ref{bound-approx-system-1})--(\ref{bound-approx-system-2}),
we can obtain  bounds for $\del_{t}(Q_{\ve}, u_{\ve})$ in $L^{1}(0, T; H^{-2})$.
In order to do this, we need to estimate the $H^{-2}$--norm of each of the other terms in
system (\ref{approx-system*}), aside from $(\del_{t}Q_\ve , \del_{t}u_{\ve})$.
Since the estimates are similar to each other, we just show some tricky ones in the following:
\begin{align*}
&\|\lambda |Q|R_{\ve}D\|_{H^{-2}}=\sup_{\varphi\in S^{d}_{0}, \|\varphi\|_{H^{2}_{0}}\leq 1}(\lambda |Q|R_{\ve}D, \varphi)
  \leq C\|\na u\|_{L^2}\|Q\|_{L^2}\|\varphi\|_{L^\infty} \\
&\quad \leq C\|\na u\|_{L^{2}}\|Q\|_{L^{2}}\|\varphi\|_{H^{2}_{0}}\leq  C\|\na u\|_{L^2}\| Q\|_{L^2},
\end{align*}
\begin{align*}
&\|\ve \pa \na \cdot R_{\ve}(\na R_{\ve}u|\na R_{\ve}u|^{2})\|_{H^{-2}}
  =\sup_{ \|\psi\|_{H^{2}_{0}}\leq 1, \rm{div}(\psi)=0}(\ve \pa \na \cdot R_{\ve}(\na R_{\ve}u|\na R_{\ve}u|^{2}), \psi)\\
&\quad =\sup_{ \|\psi\|_{H^{2}_{0}}\leq 1, \rm{div}(\psi)=0}(\ve \na R_{\ve}u|\na R_{\ve}u|^{2}, R_{\ve}\na \psi)
\leq C\ve \|\na R_{\ve}u\|_{L^2}\|\na R_{\ve} u\|^{2}_{L^4}\\
&\quad \leq C\ve \|\na u\|_{L^2}\|\na R_{\ve}u\|_{L^4}^2,
\end{align*}
\begin{align*}
&\|\lambda \na \cdot \pa R_{\ve} (|Q|H)\|_{H^{-2}}
=\sup_{ \|\psi\|_{H^{2}_{0}}\leq 1, \rm{div}(\psi)=0}(\lambda \na \cdot \pa R_{\ve} (|Q|H), \psi)\\
&\quad =\sup_{ \|\psi\|_{H^{2}_{0}}\leq 1, \rm{div}(\psi)=0}(\lambda |Q|(\D Q-a Q+b(Q^2-\frac{\tr (Q^2)}{d}\mathrm{I}_d)-cQ\,\tr (Q^2)), R_{\ve}\na \psi)\\
&\quad  \leq  C\|Q\|_{L^2}\big(\|\D Q\|_{L^2}+\|Q\|_{L^2}+\|Q\|_{L^4}^{2}+\|Q\|^{3}_{L^6}\big),
\end{align*}
and
$$
\|\kappa \pa \na \cdot R_{\ve} Q\|_{H^{-2}}=\sup_{ \|\psi\|_{H^{2}_{0}}\leq 1, \rm{div}(\psi)=0}(\kappa \pa \na \cdot R_{\ve} Q, \psi)\leq C\|Q\|_{L^2}.
$$
By the Sobolev interpolation inequality (Lemma \ref{gn-inequality}), one has
\begin{equation}\label{na-Q-L3-embedding}
\|\na Q\|_{L^3}\leq C \|{D^2} Q\|_{L^{2}}^{\frac{1}{2}}\|Q\|^{\frac{1}{2}}_{L^6}.
\end{equation}
From the above inequality, we have
\begin{align*}
&\|\ve \pa R_{\ve}(\del_{\a}Q_{\b \gamma}(R_{\ve}u\cdot \na Q_{\b \gamma})|R_{\ve}u\cdot \na Q|)\|_{H^{-2}}\\
&=\sup_{ \|\psi\|_{H^{2}_{0}}\leq 1, \rm{div}(\psi)=0}(\ve \pa R_{\ve}(\del_{\a}Q_{\b \gamma}(R_{\ve}u\cdot \na Q_{\b \gamma})|R_{\ve}u\cdot \na Q|), \psi)\\
&=\sup_{ \|\psi\|_{H^{2}_{0}}\leq 1, \rm{div}(\psi)=0}(\ve \del_{\a}Q_{\b \gamma}(R_{\ve}u\cdot \na Q_{\b \gamma})|R_{\ve}u\cdot \na Q|, R_{\ve} \psi)\\
&\leq C\ve \|\na Q\|_{L^3}\|R_{\ve}u\cdot \na Q\|^{2}_{L^3}\|\psi\|_{L^\infty}\\
&\leq C\ve \|{D^2} Q\|^{\frac{1}{2}}_{L^2}\| Q\|^{\frac{1}{2}}_{L^6}\|R_{\ve}u\cdot \na Q\|_{L^3}^2 \|\psi\|_{L^\infty}\\
&\leq C\ve \big(\|R_{\ve}u\cdot \na Q\|_{L^3}^3 + \|{D^2} Q\|^{2}_{L^2}+\| Q\|^{6}_{L^6}\big).
\end{align*}

From all the above estimates, together with the uniform bounds (\ref{bound-approx-system-1})--(\ref{bound-approx-system-2}),
we can conclude that $(\del_{t}u_{\ve}, \del_{t}Q_{\ve})\in L^{1}(0, T; H^{-2})$.
Then, by the classical Aubin-Lions compactness lemma (Lemma \ref{al-lemma}),
we know that there exist $Q\in L_{\rm loc}^{\infty}(\R_{+}; H^{1})\cap L_{\rm loc}^{2}(\R_{+}; H^{2})$
and  $u\in L_{\rm loc}^{\infty}(\R_{+}; L^{2})\cap L_{\rm loc}^{2}(\R_{+}; H^{1})$ such that,
subject to a subsequence, we have
\be \label{sub-conv-2}
\begin{split}
&Q_{\ve}\rightharpoonup Q \quad \mbox{in}\; L^{2}(0, T; H^{2}),\quad Q_{\ve}\rightarrow Q \quad \mbox{in} \; L^{2}(0, T; H_{\rm loc}^{2-\delta})
  \quad \mbox{for any}\; \delta \in(0, 4),\\
&Q_{\ve}(t)\rightharpoonup Q(t) \quad \mbox{in}\; H^{1} \quad \mbox{for any}\; t\in \R_{+},\\
&Q_{\ve}\rightharpoonup Q \quad \mbox{in}\; L^{p}(0, T; H^{1}),\quad Q_{\ve}\rightarrow Q \quad \mbox{in} \; L^{p}(0, T; H_{\rm loc}^{1-\delta})
  \quad \mbox{for any}\; \delta \in(0, 3),  p\in [2, \infty) ,\\
&u_{\ve}\rightharpoonup u \quad \mbox{in}\; L^{2}(0, T; H^{1})\quad \mbox{and}
  \quad u_{\ve}\rightarrow u \quad \mbox{in}\; L^{2}(0, T; H_{\rm loc}^{1-\delta}) \quad \mbox{for any}\; \delta \in (0, 3),\\
&u_{\ve}(t)\rightharpoonup u(t)\quad \mbox{in} \; L^{2} \quad \mbox{for any}\; t\in \R_{+}.
\end{split}
\en

With the above result, we can pass to the limit in the weak solution $(Q_{\ve}, u_{\ve})$ of system (\ref{approx-system*}),
as $\ve \to 0$, to obtain a weak solution $(Q, u)$ of system (\ref{Qtensor}) satisfying (\ref{weak-solution-Q})--(\ref{weak-solution-u}).
In the following, we focus on some terms that are not so easy to deal with.

First, we observe that
\begin{equation}\label{Q-ve-1order}
\begin{split}
&\del_{\b}(Q_{\ve})_{\gamma \delta}R_{\ve}(\del_{\b}\psi_{\a})- \del_{\b}Q_{\gamma \delta}\del_{\b}\psi_{\a}\\
&=\del_{\b}(Q_{\ve})_{\gamma \delta}\big(R_{\ve}(\del_{\b}\psi_{\a})-\del_{\b}\psi_{\a}\big)
  +\big(\del_{\b}(Q_{\ve})_{\gamma \delta}- \del_{\b}Q_{\gamma \delta}\big)\del_{\b}\psi_{\a}\\
&=\i_{\ve}^1 +\i_{\ve}^2
\end{split}
\end{equation}
converges to zero strongly in $L^{2}(0, T; L^2)$. This is owing to
the fact that $\i_{\ve}^1$ converges to zero strongly in $L^2 (0, T; L^2)$,
since $R_{\ve}(\del_{\b}\psi_{\a})-\del_{\b}\psi_{\a}$ converges strongly to zero in any $L^{p}(0, T; L^{q})$ ($\psi$ is compactly
supported and smooth) and $Q_{\ve}$ is uniformly bounded in $L^{\infty}(0, T; H^1)$,
and the fact that $\i_{\ve}^2$ converges to zero strongly in $L^2 (0, T; L^2)$,
since $Q_{\ve}$ converges to $Q$ strongly in $L^2 (0, T; H^{2-\delta}_{\rm loc})$
and $\psi$ is compactly supported and smooth.\\%
Combining the above facts with the weak convergence of $Q_{\ve}$ in $L^{2}(0, T; H^2)$ in (\ref{sub-conv-2}) yields
\begin{align*}
\int_{0}^{\infty}\int_{\R^{d}}R_{\ve}\big(\del_{\a}(Q_{\ve})_{\gamma \dl}\del_{\b}(Q_{\ve})_{\gamma \dl}\big)\del_{\b}\psi_{\a}dxdt
&=\int_{0}^{\infty}\int_{\R^{d}}\del_{\a}(Q_{\ve})_{\gamma \dl}\del_{\b}(Q_{\ve})_{\gamma \dl}R_{\ve}\del_{\b}\psi_{\a}dxdt\\
&\rightarrow \int_{0}^{\infty}\int_{\R^{d}}\del_{\a}Q_{\gamma \dl}\del_{\b}Q_{\gamma \dl}\del_{\b}\psi_{\a}dxdt.\\
\end{align*}

Moreover, from the strong convergence of $Q_{\ve}$ in $L^{p}(0, T; H^{1-\delta}_{\rm loc})$ for $p\in [1, \infty)$ in (\ref{sub-conv-2})
and the uniform bound of $Q_{\ve}$ in $L^{6}(0, T; L^6)$ in (\ref{bound-approx-system-1}), we have
\begin{equation*}
|Q_{\ve}|^{3}Q_{\ve}\rightarrow |Q|^{3}Q, \quad
|Q_{\ve}|Q_{\ve}^{2}\rightarrow |Q|Q^{2} \qquad\,\, \mbox{in} \,\,\, L^{1}(0, T; L^{1}_{\rm loc}).
\end{equation*}
As a result, we have the following convergence:
\begin{align}
&\la \int_{0}^{\infty}\int_{\R^{d}}R_{\ve}(|Q_{\ve}|H^{\ve}): \na \psi\, dx dt
=\la \int_{0}^{\infty}\int_{\R^{d}}|Q_{\ve}|H^{\ve}: R_{\ve}\na \psi\, dx dt\nonumber\\
&=\la \int_{0}^{\infty}\int_{\R^{d}}|Q_{\ve}|\Big(\D Q_{\ve}-aQ_{\ve}+b\big(Q_{\ve}^{2}+\frac{\tr((Q_{\ve})^{2})}{d}\mathrm{I}_d\big)
-cQ_{\ve}\tr((Q_{\ve})^2)\Big): R_{\ve}\na \psi\, dx dt\\
&\rightarrow \la \int_{0}^{\infty}\int_{\R^{d}}|Q|H : \na \psi \, dx dt.\nonumber
\end{align}
Finally, from the uniform bounds of $\ve \|R_{\ve}u\cdot \na Q\|^{3}_{L^{3}}$ in (\ref{bound-approx-system-2}), along with (\ref{na-Q-L3-embedding}),
we have
\begin{equation}
\begin{split}
&\ve\int_{0}^{\infty}\int_{\R^d} R_{\ve}\big(|R_{\ve}u\cdot \na Q|(R_{\ve}u\cdot \na Q_{\b \gamma})\del_{\a}Q_{\b \gamma}\big) \psi_{\a}\, dx dt\\
&=\ve\int_{0}^{T}\int_{\R^d} \big(|R_{\ve}u\cdot \na Q|(R_{\ve}u\cdot \na Q_{\b \gamma})\del_{\a}Q_{\b \gamma}\big) R_{\ve}\psi_{\a}\, dx dt\\
&\leq \ve \int_{0}^{T}\|R_{\ve}u\cdot \na Q\|_{L^{3}}^{2}\|\na Q\|_{L^3}\|R_{\ve}\psi\|_{L^\infty}dt\\
&\leq C\ve \int_{0}^{T}\|R_{\ve}u\cdot \na Q\|_{L^{3}}^{2}\|{D^2} Q\|_{L^2}^{\frac{1}{2}}\| Q\|_{L^{6}}^{\frac{1}{2}}dt\\
&\leq C\ve  \|R_{\ve}u\cdot \na Q\|_{L^{3}(0, T; L^3)}^{2}\|Q\|_{L^{6}(0, T; L^6)}^{\frac{1}{2}}\|Q\|_{L^{2}(0, T; H^{2})}^{\frac{1}{2}}\\
&\rightarrow 0.
\end{split}
\end{equation}
Similarly to the above estimate, by using the uniform bound of $\ve \|\na R_{\ve}u\|^{4}_{L^{4}}$ in (\ref{bound-approx-system-2}),
we also conclude that $\ve \int R_{\ve}\big(\na R_{\ve}u|\na R_{\ve}u|^2\big): \na \psi \, dx dt\to 0$ as $\ve \to 0$.

Then the proof of Theorem \ref{weak-solution-thm} is completed.

\medskip

We remark that, in the proof above, because of the active terms in our system (\ref{Qtensor}),
in order to obtain the convergence of the approximate solutions,
we need to establish a higher integrability of $Q$
in time and use the Sobolev interpolation inequalities to achieve the uniform $H^{-2}$--estimates
for the extra terms added to the approximate system, which is different from the passive case
treated in \cite{P-Z-2011}.

\bigskip

\section{Higher Regularity In Two Dimensions}\label{2D-regularity}

In this section, we prove that, in the two-dimensional case,
system (\ref{Qtensor}) has solutions with higher regularity,
provided with sufficiently regular initial data.
The result is stated in Theorem \ref{regularity-thm-2d}.
We   mention that   we use the Littlewood-Paley decomposition to help us improve
the higher regularity of the solution.
Of course, we can also obtain the higher regularity by differentiating the equations $k\geq 1$ times
in system (\ref{Qtensor}).
However, this requires the initial data $(\bar{Q}, \bar{u})$ to be at least
in  $H^2 \times H^1$,
rather than in $H^{s+1} \times H^s$ for $s>0$ in the Littlewood-Paley method.

\begin{Remark}
For any traceless, symmetric, $2\times 2$ matrix $Q$,
\begin{equation*}
Q^2 -\frac{\tr (Q^2)}{2}\mathrm{I}_{2} =0,
\end{equation*}
which means that $H$ reduces to a simpler form
\begin{equation*}
H=\D Q-a Q-c\,Q\,\tr (Q^2) \qquad \mbox{for $x\in \R^2$}.
\end{equation*}
\end{Remark}

\begin{Theorem}\label{regularity-thm-2d}
Let $s>0$ and $(\bar{Q}, \bar{u})\in H^{s+1}(\R^2)\times H^{s} (\R^2)$.
There exists a global solution $(Q(t, x), u(t, x))$ to system \eqref{Qtensor}
with the initial conditions:
\be
Q(0, x)=\bar{Q}(x), \qquad u(0, x)=\bar{u}(x)
\en
such that
$$
Q\in L_{\rm loc}^2 (\R_+ ; H^{s+2}(\R^2))\cap L_{\rm loc}^{\infty} (\R_+ ; H^{s+1}(\R^2)),
$$
$$
u \in L_{\rm loc}^2 (\R_+ ; H^{s+1}(\R^2))\cap L_{\rm loc}^{\infty} (\R_+ ; H^{s}(\R^2)),
$$
and
\be\label{higher-regularity}
\|\na Q(t, \cdot)\|_{H^{s}}^{2}+\|u(t, \cdot)\|_{H^{s}}^{2}\leq Ce^{e^{Ct}},
\en
where constant $C$ depends on $\bar{Q}, \bar{u}, a, b, c$, and $\Gamma$.
\end{Theorem}

To prove this theorem, we restrict ourselves to system (\ref{Qtensor}) in two spatial dimensions
and give the {\it a priori} estimates for the smooth solutions of  this system.
Of course, the same estimates independent of $\ve$ can also be obtained,
if we use the modified system (\ref{approx-system*}),
whose solutions are smooth as we mentioned in Remark \ref{remark-1}.
Moreover, we use the Littlewood-Paley decomposition to obtain the {\it a priori} estimates of the solution.
Firstly, we apply $\D_q$ (see Appendix A for the notations), with $q\in \mathbb{N}$,
to the equations in system (\ref{Qtensor})
to obtain the estimates of high frequencies of the solution.
Secondly, by applying $S_{0}$ to this system, we obtain the estimates of low frequencies.
Finally, we achieve the high regularity of the solution
by combining the high and low frequencies together and by using Gronwall's inequalities (Lemma \ref{Gronwall}).

\begin{proof} We begin with the estimates of high frequencies.
Applying $\D_{q}$ to the first equation in (\ref{Qtensor}),
using Bony's paraproduct decomposition, multiplying the equation by $-\D \D_q Q$,
taking the trace, and integrating by parts over $\R^2$, we have
\begin{align} \label{high-frequency-Q}
&\frac{1}{2}\frac{d}{dt}\|\nabla \D_q Q\|_{L^2}^2 +\Gamma \|\D \D_q Q\|_{L^2}^2
+(\D_{q} \Omega S_{q-1}Q -S_{q-1}Q\D_q \Omega, \D \D_q Q)\nonumber\\
&=(\D_q (u\nabla Q), \D \D_q Q)-\sum_{|q^{'} -q|\leq 5}([\D_q ; S_{q^{'} -1}Q_{\gamma \beta}]\D_{q^{'}} \Omega_{\alpha \gamma}, \D \D_q Q_{\alpha \beta})\nonumber\\
&\quad -\sum_{|q^{'} -q|\leq 5}((S_{q^{'}-1}Q_{\gamma \b}-S_{q-1}Q_{\gamma \b})\D_q \D_{q^{'}} \Omega_{\alpha \gamma}, \D \D_q Q_{\alpha \beta})\nonumber\\
&\quad -\sum_{q^{'} >q-5}(\D_q (S_{q^{'} +2}\Omega_{\a \gamma}\D_{q^{'}} Q_{\gamma \b}), \D \D_q Q_{\alpha \beta})\\
&\quad +\sum_{|q^{'} -q|\leq 5}([\D_q ; S_{q^{'} -1}Q_{\a \gamma}]\D_{q^{'}} \Omega_{\gamma \b}, \D \D_q Q_{\alpha \beta})\nonumber\\
&\quad +\sum_{|q^{'} -q|\leq 5}((S_{q^{'}-1}Q_{\a \gamma}-S_{q-1}Q_{\a \gamma})\D_q \D_{q^{'}} \Omega_{\gamma \b}, \D \D_q Q_{\alpha \beta})\nonumber\\
&\quad +\sum_{q^{'} >q-5}(\D_q (S_{q^{'} +2}\Omega_{\gamma \b}\D_{q^{'}} Q_{\a \gamma}), \D \D_q Q_{\alpha \beta})\nonumber\\
&\quad-\la (\D_q (|Q|D), \D \D_q Q)+\Gamma a (\D_q Q, \D \D_q Q)+\Gamma c(\D_q (Q\tr (Q^2)), \D \D_q Q)\nonumber\\
&=\sum_{i=1}^{10} \i_{i}.\nonumber
\end{align}
Similarly, we can apply $\D_q$ to the second equation in (\ref{Qtensor}), use Bony's paraproduct decomposition again,
multiply the  equation by $\D_q u$, and integrate by parts over $\R^2$ to obtain
\begin{align} \label{high-frequency-u}
&\!\!\!\! \frac{1}{2}\frac{d}{dt}\|\D_q u\|_{L^2}^2+\mu \|\na \D_q u\|_{L^2}^2 +(S_{q-1}Q\D_q \D Q-\D_{q}\D Q S_{q-1}Q,\D_q \na u)\nonumber \\
=&-(\D_q (u\cd \na u), \D_q u)+(\D_q (\del_{\a}Q_{\gamma \dl}\del_{\b}Q_{\gamma \dl}), \D_q \del_{\b}u_{\a})\nonumber\\
&-\sum_{|q^{'} -q|\leq 5}([\D_q ; S_{q^{'} -1}Q_{\a \gamma}]\D_{q^{'}} \D Q_{\gamma \b}, \D_q \del_{\b}u_{\a})\nonumber\\
&-\sum_{|q^{'} -q|\leq 5}((S_{q^{'}-1}Q_{\a \gamma}-S_{q-1}Q_{\a \gamma})\D_q \D_{q^{'}} \D Q_{\gamma \b}, \D_q \del_{\b}u_{\a})\nonumber\\
&-\sum_{q^{'} >q-5}(\D_q (S_{q^{'} +2}\D Q_{\gamma \b}\D_{q^{'}} Q_{\a \gamma}), \D_q \del_{\b}u_{\a})\\
&+\sum_{|q^{'} -q|\leq 5}([\D_q ; S_{q^{'} -1}Q_{\gamma \b}]\D_{q^{'}} \D Q_{\a \gamma}, \D_q \del_{\b}u_{\a})\nonumber\\
&+\sum_{|q^{'} -q|\leq 5}((S_{q^{'}-1}Q_{\gamma \b}-S_{q-1}Q_{\gamma \b})\D_q \D_{q^{'}} \D Q_{\a \gamma}, \D_q \del_{\b}u_{\a})\nonumber\\
&+\sum_{q^{'} >q-5}(\D_q (S_{q^{'} +2}\D Q_{\a \gamma}\D_{q^{'}} Q_{\gamma \b}), \D_q \del_{\b}u_{\a})+\la (\D_q (|Q|\D Q), \na \D_q u)\nonumber\\
&-a\la (\D_q (|Q|Q), \na \D_q u)-c\la (\D_q (|Q|Q\,\tr(Q^2)), \na \D_q u)-\kappa(\D_q Q, \na \D_q u)\nonumber\\
=&\sum_{j=1}^{12}\j_{j}.\nonumber
\end{align}
Adding up (\ref{high-frequency-Q})--(\ref{high-frequency-u}) and using Lemma \ref{estimate-matrix}, we have
\be \label{high-frequency}
\begin{split}
\frac{1}{2}\frac{d}{dt}\big(\|\na \D_q Q\|_{L^2}^2 +\|\D_q u\|_{L^2}^2\big)
+\mu \|\na \D_q u\|_{L^2}^2 +\Gamma \|\D \D_q Q\|_{L^2}^2 =\sum_{i=1}^{10} \i_{i} +\sum_{j=1}^{12} \j_{j}.
\end{split}
\en
Set  $$\va (t):=\|\na Q\|_{H^{s}}^2 +\|u\|_{H^{s}}^2,\quad \va_1 (t):=\|S_0 \na Q\|_{L^2}^2 +\|S_0 u\|_{L^2}^2,\quad
 \va_2 (t):=\va (t)-\va_{1}(t),$$
  with $\va_1$ and $\va_2$ representing the low-frequency part and high-frequency part of $\va$, respectively.
Then (\ref{high-frequency}) leads to the following estimate (see Appendix B for the details):

\be \label{high-freq-Q-u}
\begin{split}
&\!\!\!\! \frac{1}{2}\frac{d}{dt}\va_{2}(t)+\sum_{q\in \mathbb{N}}2^{2qs}\big(\mu \|\D_q \na u\|_{L^2}^2 +\Gamma \|\D_q \D Q\|_{L^2}^2\big)\\
\leq & C\big(1+\|u\|_{L^{2}}^{2}\|\na u\|_{L^{2}}^{2}+\|\na Q\|_{L^{2}}^{2}\|\D Q\|_{L^{2}}^{2}
 +\|Q\|_{L^{2}}^{2}+\|Q\|_{L^{4}}^{4}+\|Q\|_{L^{6}}^{6}\big)\va (t)\\
&+\frac{\Gamma}{4}\|\D Q\|_{H^{s}}^2 +\frac{\mu }{4}\|\na u\|_{H^{s}}^2.
\end{split}
\en

Next, we estimate the low frequencies. Applying $S_0$ to the first equation in (\ref{Qtensor}), multiplying by $-S_0 \D Q$,
taking the trace, and integrating by parts over $\R^2$, we have
\begin{align}
&\frac{1}{2}\frac{d}{dt}\|\na S_0 Q\|_{L^{2}}^{2}+\Gamma \|S_0 \D Q\|_{L^{2}}^{2}\nonumber\\
&=(S_0 (u\cdot \na Q), S_0 \D Q)+(S_0 (Q\O -\O Q), S_0 \D Q)\nonumber\\
&\quad -\la (S_0 (|Q|\O), S_0 \D Q)+a\Gamma (S_0 Q, S_0 \D Q)+c\Gamma (S_0 (Q\tr(Q^2)), S_0 \D Q)\nonumber\\
&\leq C\|u\|_{L^4}\|\na Q\|_{L^4}\|\s \D Q\|_{L^2}+\|S_0 (Q\O -\O Q)\|_{L^{2}}\|\s \D Q\|_{L^2}\nonumber\\
&\quad+C\|S_{0}(|Q|\O)\|_{L^{2}}\|\s \D Q\|_{L^2}+C\|\s \na Q\|_{L^2}^2 +C\|Q\|_{L^6}^3 \|\s \D Q\|_{L^2}  \nonumber \\
&\leq C\|u\|_{L^{2}}^{\frac{1}{2}}\|\na u\|_{L^{2}}^{\frac{1}{2}}\|\na Q\|_{L^{2}}^{\frac{1}{2}}\|\D Q\|_{L^{2}}^{\frac{1}{2}}\|\s \D Q\|_{L^2}+\|S_0 (Q\O -\O Q)\|_{L^{1}}\|\s \D Q\|_{L^2}\label{low-freq-Q}\\
&\quad+C\|S_{0}(|Q|\O)\|_{L^{1}}\|\s \D Q\|_{L^2}+C\|\s \na Q\|_{L^2}^2 +C\|Q\|_{L^6}^3 \|\s \D Q\|_{L^2} \nonumber\\
&\leq C\|u\|_{L^{2}}^{\frac{1}{2}}\|\na u\|_{L^{2}}^{\frac{1}{2}}\|\na Q\|_{L^{2}}^{\frac{1}{2}}\|\D Q\|_{L^{2}}^{\frac{1}{2}}\|\s \D Q\|_{L^2}+C\|Q\|_{L^{2}}\|\na u\|_{L^{2}}\|\s \D Q\|_{L^2}\nonumber\\
&\quad+C\|\s \na Q\|_{L^2}^2 +C\|Q\|_{L^6}^3 \|\s \D Q\|_{L^2} \nonumber\\
&\leq \frac{\Gamma}{4}\|\s \D Q\|_{L^{2}}^{2} +C\|\na Q\|_{H^{s}}^2 +C\|u\|_{L^{2}}^{2}\|\na u\|_{L^{2}}^{2}+C\|\na Q\|_{L^{2}}^{2}\|\D Q\|_{L^{2}}^{2}\nonumber\\
&\quad+C\|Q\|_{L^{2}}^{2}\|\na u\|_{L^{2}}^{2}+C\|Q\|_{L^6}^{6}.\nonumber
\end{align}
Then we apply $\s$ to the second equation in (\ref{Qtensor}), multiply by $\s u$, and integrate by parts over $\R^2$ to obtain
\begin{align}\label{low-freq-u}
& \frac{1}{2}\frac{d}{dt}\|\s u\|_{L^2}^2 +\mu \|\s \na u\|_{L^2}^2 \nonumber \\
&=-(\s (u\cdot \na u), \s u)+(\s (\na Q \odot \na Q), \s \na u)-(\s \na \cdot (Q \D Q-\D Q Q), \s u)\nonumber\\
&\quad -\l (\s \na \cdot(|Q|(\D Q- aQ -cQ\,\tr(Q^2))), \s u)+\kappa(\s \na \cdot Q, \s u)\nonumber \\
&\leq \|\s (u\cdot \na u)\|_{L^{2}}\|\s u\|_{L^{2}}+\|\s (\na Q \odot \na Q)\|_{L^{2}} \|\s \na u\|_{L^{2}}\nonumber \\
&\quad +|\l| \|\s (|Q|(\D Q- aQ -cQ\,\tr(Q^2)))\|_{L^{2}} \|\s \na u\|_{L^{2}}\nonumber \\
&\quad +\|\s (Q \D Q-\D Q Q)\|_{L^{2}} \|\s \na u\|_{L^{2}}+|\kappa| \|\s \na \cdot Q\|_{L^{2}} \|\s u\|_{L^{2}}\nonumber \\
&\leq C\|\s (u\cdot \na u)\|_{L^{1}}\|\s u\|_{L^{2}}+C\|\s (\na Q \odot \na Q)\|_{L^{1}} \|\s \na u\|_{L^{2}} \\
&\quad +C\|\s (|Q|(\D Q- aQ -cQ\,\tr(Q^2)))\|_{L^{1}} \|\s \na u\|_{L^{2}}\nonumber \\
&\quad +C\|\s (Q \D Q-\D Q Q)\|_{L^{1}} \|\s \na u\|_{L^{2}}+C \|\s \na \cdot Q\|_{L^{2}} \|\s u\|_{L^{2}}\nonumber \\
&\leq C\|u\|_{L^{2}}\|\na u\|_{L^{2}}\|\s u\|_{L^{2}}+C\|\na Q\|_{L^{2}}^{2}\|\s \na u\|_{L^{2}}\nonumber \\
&\quad+C\big(\|Q\|_{L^{2}}\|\D Q\|_{L^{2}} +\|Q\|_{L^{2}}^{2} +\|Q\|_{L^{4}}^4\big)\|\s \na u\|_{L^{2}}+C\|\s \na Q\|_{L^2} \|\s u\|_{L^2}\nonumber \\
&\leq \frac{\mu }{4}\|\s \na u\|_{L^{2}}^{2}+C\|u\|_{H^{s}}^{2}+C\|\na Q\|_{H^{s}}^{2}+C\|u\|_{L^{2}}^{2}\|\na u\|_{L^{2}}^{2}+C\|\na Q\|_{L^{2}}^{4}\nonumber \\
&\quad+C\big(\|Q\|_{L^{2}}^{2}\|\D Q\|_{L^{2}}^{2} +\|Q\|_{L^{2}}^{4} +\|Q\|_{L^{4}}^{8}\big).
\nonumber
\end{align}
We add (\ref{low-freq-Q})--(\ref{low-freq-u}) to obtain
\be \label{low-freq-Q-u}
\begin{split}
&\frac{1}{2}\frac{d}{\mathrm{d}t}\va_{1}(t)+\frac{3\mu }{4}\|\s \na u\|_{L^2}^2 +\frac{3\Gamma}{4}\|\s \D Q\|_{L^2}^2 \\
&\leq C \va (t) +C\|u\|_{L^{2}}^{2}\|\na u\|_{L^{2}}^{2}+C\|\na Q\|_{L^{2}}^{2}\|\D Q\|_{L^{2}}^{2}+C\|Q\|_{L^{2}}^{2}\|\na u\|_{L^{2}}^{2}\\
&\quad +C\|Q\|_{L^{2}}^{2}\|\D Q\|_{L^{2}}^{2}
  +C\big(\|Q\|_{L^{2}}^{4} +\|Q\|_{L^{4}}^{8} +\|Q\|_{L^6}^{6} +\|\na Q\|_{L^{2}}^{4}\big).
\end{split}
\en

Finally, we derive the estimates of the high norms.
Adding (\ref{high-freq-Q-u}) and (\ref{low-freq-Q-u}), we have
\be
\begin{split}
&\!\!\!\! \frac{d}{\mathrm{d}t}\va (t)+\frac{\mu}{2} \| \na u\|_{H^s}^2 +\frac{\Gamma}{2} \|\D Q\|_{H^s}^2\leq A(t)\va (t)+B(t),
\end{split}
\en
with
\begin{align*}
A(t)=&C\big(1+\|u\|_{L^{2}}^{2}\|\na u\|_{L^{2}}^{2}+\|\na Q\|_{L^{2}}^{2}\|\D Q\|_{L^{2}}^{2}+\|Q\|_{L^{2}}^{2}+\|Q\|_{L^{4}}^{4}+\|Q\|_{L^{6}}^{6}\big),\\
B(t)=&C\|u\|_{L^{2}}^{2}\|\na u\|_{L^{2}}^{2}+C\|\na Q\|_{L^{2}}^{2}\|\D Q\|_{L^{2}}^{2}+C\|Q\|_{L^{2}}^{2}\|\na u\|_{L^{2}}^{2}\\
 &+C\|Q\|_{L^{2}}^{2}\|\D Q\|_{L^{2}}^{2} +C\big(\|Q\|_{L^{2}}^{4} +\|Q\|_{L^{4}}^{8} +\|Q\|_{L^6}^{6} +\|\na Q\|_{L^{2}}^{4}\big).
\end{align*}
From the {\it a priori} estimates in Proposition \ref{energy-estimate},
we know that both $A(t)$ and $B(t)$ belong to $L^{1}(0, T)$ and increase exponentially in time. Then, by Gronwall's inequality (Lemma \ref{Gronwall}), we can conclude that $\va (t)$ increases double exponentially in time as stated in (\ref{higher-regularity}).
\end{proof}

\section{Weak-Strong Uniqueness in Dimension Two}\label{2D-uniqueness}

In this section, we prove that a global weak solution and a strong one must coincide,
provided that they have the same initial data $(\bar{Q}, \bar{u})\in H^{s+1}(\R^{2})\times H^{s}(\R^{2})$
for $s>0$.
The result is stated in the following theorem.

\begin{Theorem} Let $(\bar{Q}, \bar{u})\in H^{s+1}(\R^{2})\times H^{s}(\R^{2})$, with $s>0$,  be the initial data.
By Theorem {\rm \ref{weak-solution-thm}}, there exists a weak solution $(Q_1 , u_1)$ of system \eqref{Qtensor} such that
\be\label{Q1u1}
Q_1 \in L_{\rm loc}^{\infty}(\R_{+}; H^1)\cap L_{\rm loc}^2 (\R_{+}; H^2 ), \quad  \quad u_1 \in L_{\rm loc}^{\infty}(\R_{+}; L^2 )\cap L_{\rm loc}^{2}(\R_{+}; H^1 ).
\en
Theorem {\rm \ref{regularity-thm-2d}} gives the existence of a strong solution $(Q_2 , u_2)$ such that
\be\label{Q2u2}
Q_2 \in L_{\rm loc}^{\infty}(\R_{+}; H^{s+1})\cap L_{\rm loc}^2 (\R_{+}; H^{s+2} ), \quad \quad u_2 \in L_{\rm loc}^{\infty}(\R_{+}; H^{s} )\cap L_{\rm loc}^{2}(\R_{+}; H^{s+1} ).
\en
Then $(Q_1 , u_1)=(Q_2, u_2)$.
\end{Theorem}

\begin{proof}
Denote $\delta Q:=Q_1 -Q_2$ and $\delta u:=u_1 -u_2$. Then $(\delta Q, \delta u)$ satisfies the following system:

\be \label{unique}
\begin{cases}
\del_{t}\dl Q+\dl u \cdot \na\dl Q-\dl \O \dl Q+\dl Q \dl \O+\dl u \cdot \na Q_2 +u_2 \cdot \na \dl Q
 +Q_2 \dl \O+\dl Q \O_2-\dl \O Q_2 -\O_2 \dl Q
\\
 \quad \quad\quad\; =\la |Q_1 |\dl D +\la (|Q_1|-|Q_2|)D_2 \\
 \quad \quad\quad \quad+\Gamma \big(\D\dl Q-a\dl Q -c(\dl Q \tr (Q_{1}^2) +Q_2 \tr(Q_1 \dl Q+\dl Q Q_2))\big), \\[1mm]
 \del_{t}\dl u+\pa (\dl u \cdot \na \dl u)\\
\quad \quad\quad \;  =\mu \D \dl u-\pa (\na \cdot (\na \dl Q \odot \na \dl Q))\\
 \quad \quad\quad \quad+\pa (\na \cdot (\dl Q\D \dl Q-\D \dl Q \dl Q))-\pa (u_2 \cdot \na \dl u +\dl u \cdot \na u_2)\\
 \quad \quad\quad \quad +\pa (\na \cdot (\dl Q \D Q_2 +Q_2 \D \dl Q -\D \dl Q Q_2 -\D Q_2 \dl Q))\\
 \quad \quad\quad \quad -\pa (\na \cdot (\na \dl Q \odot \na Q_2+\na Q_2 \odot \na \dl Q))-\la \pa (\na \cdot (|Q_1 |(\D \dl Q -a \dl Q)))\\
 \quad \quad\quad \quad -\la \pa (\na \cdot (|Q_1 |-|Q_2 |)(\D Q_2 -a Q_2))+\la c \pa \big(\na \cdot (|Q_1 |\dl Q \tr (Q_1^2))\big)\\
 \quad \quad\quad \quad +\la c \pa \big(\na \cdot (|Q_1 |Q_2 \tr(Q_1 \dl Q +\dl Q Q_2 ))\big)\\
 \quad \quad\quad \quad +\la c \pa \big(\na \cdot ((|Q_1 |-|Q_2 |)Q_2 \tr (Q_2^2))\big)+\kappa \pa (\na \cdot \dl Q).
\end{cases}
\en
Similarly to the proof of Proposition \ref{energy-inequality},
we  multiply the first equation in (\ref{unique}) by $-\D \dl Q +\dl Q$,
take the trace, and integrate by parts over $\R^2$ to obtain
\be \label{unique-Q}
\begin{split}
&\frac{1}{2}\frac{d}{dt}\big(\|\na \dl Q\|_{L^2}^2 +\|\dl Q\|_{L^2}^2\big)
 -(\dl u \cdot \na \dl Q, \D \dl Q)-(\dl Q \dl \O -\dl \O \dl Q, \D \dl Q)\\
&\quad-(\dl u \cdot \na Q_2 +u_2 \cdot \na \dl Q+\dl Q \O_2 -\O_2 \dl Q, \D \dl Q)+(\dl u \cdot \na Q_2 , \dl Q)\\
&\quad -(Q_2 \dl \O -\dl \O Q_2 , \D \dl Q)+(Q_2 \dl \O -\dl \O Q_2 , \dl Q)\\
&=-\Gamma \|\D \dl Q\|_{L^2}^2 -\Gamma \|\na \dl Q\|_{L^2}^2 -a\Gamma \|\na \dl Q\|_{L^2}^2 -a\Gamma \|\dl Q\|_{L^2}^2 \\
&\quad +\la (|Q_1 |\dl D+(|Q_1 |-|Q_2 |)D_2 , -\D \dl Q+\dl Q)+c\Gamma  (\dl Q\tr (Q_1^2) , \D \dl Q)\\
&\quad +c\Gamma (Q_2 \tr(Q_1 \dl Q +\dl Q Q_2), \D \dl Q)-c\Gamma (\dl Q |Q_1 |^2 , \dl Q) \\
&\quad -c\Gamma (Q_2 \tr (Q_1 \dl Q+\dl Q Q_2), \dl Q).
\end{split}
\en
Multiplying the second equation in (\ref{unique}) by $\dl u$ and integrating by parts over $\R^2$ yields
\be \label{unique-u}
\begin{split}
&\!\!\!\! \frac{1}{2}\frac{d}{dt}\|\dl u\|_{L^2}^2 +\mu \|\na \dl u\|_{L^2}^2 \\
=&-(\na \cdot (\na \dl Q\odot \na \dl Q), \dl u)-(\dl Q \D \dl Q-\D \dl Q\dl Q, \na \dl u^{\top})\\
&-(u_2 \cdot \na \dl u+\dl u \cdot \na u_2 , \dl u)-(Q_2 \D \dl Q-\D \dl QQ_2 , \na \dl u^{\top})\\
&-(\dl Q \D Q_2  -\D Q_2 \dl Q, \na \dl u^{\top})+(\na \dl Q \odot \na Q_2 +\na Q_2 \odot \na \dl Q, \na \dl u^{\top})\\
&+\la (|Q_1 |\D \dl Q, \na \dl u)-a\la (|Q_1 |\dl Q_1 , \na \dl u)+\la ((|Q_1 |-|Q_2 |)(\D Q_2 -a Q_2 ), \na \dl u)\\
&-c\la (|Q_1 |\dl Q \,\tr (Q_1^2) , \na \dl u)-c\la \big(|Q_1 |Q_2\, \tr(Q_1 \dl Q+\dl Q Q_2 ), \na \dl u\big)\\
&-c\la ((|Q_1 |-|Q_2 |)Q_2 \,\tr (Q_2^2), \na \dl u)-\kappa(\dl Q, \na \dl u).
\end{split}
\en
Adding (\ref{unique-Q}) and (\ref{unique-u}) together, and performing analogous cancellations to those
in the proof of Proposition \ref{energy-inequality}, we have
\begin{align*}
&\!\!\!\! \frac{1}{2}\frac{d}{dt}\big(\|\na \dl Q\|_{L^2}^2 +\|\dl Q\|_{L^2}^2 +\|\dl u\|_{L^2}^2\big)
+\mu \|\na \dl u\|_{L^2}^2 +\Gamma\big(\|\D \dl Q\|_{L^2}^2 + \|\na \dl Q\|_{L^2}^2\big) \\
=\, &(\dl u \cdot \na Q_2 +u_2 \cdot \na \dl Q, \D \dl Q)+(\dl Q \O_2 -\O_2 \dl Q, \D \dl Q)-(\dl u \cdot \na Q_2 , \dl Q)\\
&+\la (|Q_1 |\dl D, \dl Q)-\la(|Q_1 |-|Q_2 |)D_2 , \D \dl Q - \dl Q)+c\Gamma  (\dl Q\,\tr (Q_1^2) , \D \dl Q)\\
&+c\Gamma (Q_2 \tr(Q_1 \dl Q +\dl Q Q_2), \D \dl Q)-c\Gamma (\dl Q |Q_1 |^2 , \dl Q) -c\Gamma (Q_2 \tr (Q_1 \dl Q+\dl Q Q_2), \dl Q)\\
&+(\dl u \cdot \na \dl u, u_{2})-(\dl Q \D Q_2 -\D Q_2 \dl Q, \na \dl u^{\top})+(\na \dl Q \odot \na Q_2 +\na Q_2 \odot \na \dl Q, \na \dl u^{\top})\\
&-a\la (|Q_1 |\dl Q , \na \dl u)+\la ((|Q_1 |-|Q_2 |)(\D Q_2 -a Q_2 ), \na \dl u)-c\la (|Q_1 |\dl Q\, \tr (Q_1^2), \na \dl u)\\
&-c\la (|Q_1 |Q_2 \,\tr(Q_1 \dl Q+\dl Q Q_2 ), \na \dl u)-c\la ((|Q_1 |-|Q_2 |)Q_2\, \tr(Q_2^2), \na \dl u)\\
&-\kappa(\dl Q, \na \dl u)-a\Gamma \big(\|\na \dl Q\|_{L^2}^2 +\|\dl Q\|_{L^2}^2\big)\\
\leq \, &\|\D \dl Q\|_{L^2}\big(\|\dl u\|_{L^2}\| \na Q_2 \|_{L^\infty}+\|\na \dl Q\|_{L^2}\|u_2 \|_{L^{\infty}}
    +2\|\dl Q\|_{L^{\frac{2}{s}}}\| \O_2 \|_{L^{\frac{2}{1-s}}}\big)\\
&+\|\dl Q\|_{L^2}\|\dl u\|_{L^2}\|\na Q_2 \|_{L^{\infty}}+|\la| \|\dl Q\|_{L^2}\|Q_1 \|_{L^{\infty}}\|\na \dl u\|_{L^2}\\
&+|\la|\big(\|\D \dl Q\|_{L^2}+\|\dl Q\|_{L^{2}}\big) \|\dl Q\|_{L^{\frac{2}{s}}}\| D_2 \|_{L^{\frac{2}{1-s}}}+c\Gamma \|\D \dl Q\|_{L^2}\|\dl Q\|_{L^4}\|Q_1 \|_{L^8}^2 \\
&+c\Gamma \|\D \dl Q\|_{L^2}\|Q_2 \|_{L^{\infty}}\big(\|Q_1 \|_{L^4}+\|Q_2 \|_{L^4}\big)\|\dl Q\|_{L^4}+c\Gamma \|\dl Q\|_{L^4}^{2}\|Q_1 \|_{L^2}\|Q_1 \|_{L^{\infty}}\\
&+c\Gamma \|\dl Q\|_{L^4}^{2}\|Q_2 \|_{L^2}\big(\|Q_1 \|_{L^{\infty}}+\|Q_2 \|_{L^{\infty}}\big)+\|u_2 \|_{L^{\infty}}\|\dl u\|_{L^2}\|\na \dl u\|_{L^2}\\
&+2\|\dl Q\|_{L^{\frac{2}{s}}}\|\D Q_2 \|_{L^{\frac{2}{1-s}}}\| \na \dl u\|_{L^2}+2\|\na \dl Q\|_{L^2}\|\na \dl u\|_{L^2}\|\na Q_2 \|_{L^{\infty}}\\
&+|a\la| \|\na \dl u\|_{L^2}\|\dl Q\|_{L^2}\|Q_1 \|_{L^{\infty}}
  +C\|\na \dl u\|_{L^2}\big(\|\dl Q\|_{L^{\frac{2}{s}}}\|\D Q_2 \|_{L^{\frac{2}{1-s}}}+\|\dl Q\|_{L^{2}} \|Q_2 \|_{L^{\infty}}\big)\\
&+c|\la| \|\na \dl u\|_{L^2}\|\dl Q\|_{L^4}\big(\|Q_1 \|_{L^{\infty}}(\|Q_1 \|^{2}_{L^8}+\|Q_2 \|^{2}_{L^8})+\|Q_2 \|_{L^8}^2 \|Q_2 \|_{L^{\infty}}\big)\\
&+|\kappa| \|\na \dl u\|_{L^2}\|\dl Q\|_{L^2}-a\Gamma \big(\|\na \dl Q\|_{L^2}^2 +\|\dl Q\|_{L^2}^2\big)\\
\leq \, & \frac{\Gamma}{2}\|\D \dl Q\|_{L^2}^{2}+\frac{\mu}{2}\|\na \dl u\|_{L^2}^{2}+C\big(\|\na Q_2 \|_{L^\infty}^{2}+\|u_2 \|_{L^\infty}^{2}\big)\|\dl u\|_{L^2}^{2}\\
&+C \big(1+\|Q_2 \|_{L^{\infty}}^{2}+\|Q_1 \|_{L^{\infty}}^{2}\big)\|\dl Q\|_{L^2}^{2}+C(1+\|u_2 \|_{L^{\infty}}^{2}+\|\na Q_2 \|_{L^{\infty}}^{2})\|\na \dl Q\|_{L^2}^{2}\\
&+C\Gamma \big(\|Q_1 \|_{L^8}^{4}+\|Q_2 \|_{L^{\infty}}^{2}(\|Q_1 \|_{L^4}^{2}+\|Q_2 \|_{L^4}^{2})+\|Q_1 \|_{L^8}^{4}\|Q_1 \|_{L^{\infty}}^{2}+\|Q_2 \|_{L^8}^{4}\|Q_2 \|_{L^{\infty}}^{2}\\\
&\qquad\quad +\|Q_1 \|_{L^{\infty}}^{2}(\|Q_1 \|_{L^8}^{4}+\|Q_2 \|_{L^8}^{4})+(\|Q_1 \|_{L^{\infty}}+\|Q_2 \|_{L^{\infty}})(\|Q_1 \|_{L^{2}}+\|Q_2 \|_{L^{2}})\big)\|\dl Q\|_{L^4}^2 \\
&+C\big(\|\na u_{2}\|_{L^{\frac{2}{1-s}}}^{2}+\|\D Q_{2}\|_{L^{\frac{2}{1-s}}}^{2}\big)\|\dl Q\|_{L^{\frac{2}{s}}}^{2}.
\end{align*}
Since we restrict ourselves to the two-dimensional case,
$\|\dl Q\|^{2}_{L^4}$ and $\|\dl Q\|_{L^{\frac{2}{s}}}^{2}$ can be controlled by $\|\dl Q\|^{2}_{L^2}+\|\na \dl Q\|^{2}_{L^2}$.
Moreover, we know from the imbedding theorem that
$$
\|\na u_{2}\|_{L^{\frac{2}{1-s}}}\leq C \|u_{2}\|_{H^{s+1}}, \quad
 \|\D Q_{2}\|_{L^{\frac{2}{1-s}}}\leq C \|Q_{2}\|_{H^{s+2}}.
$$
In addition, from the conditions of $Q_1, u_1 , Q_2$, and $u_2$, {\it i.e.}, (\ref{Q1u1})--(\ref{Q2u2}),
we know that all the coefficients of $\|\dl u\|_{L^2}^{2}$, $\|\dl Q\|_{L^2}^{2}$, $\|\na \dl Q\|_{L^2}^{2}$,
$\|\dl Q\|_{L^4}^{2}$, and $\|\dl Q\|_{L^{\frac{2}{s}}}^{2}$ are integrable with respect to time.
Therefore,  we can use Gronwall's inequality (Lemma \ref{Gronwall}) to conclude the uniqueness of the solution.
\end{proof}

\bigskip

\appendix
\section{Some Basic Theories and Lemmas}\label{appendix-a}

In this section, we review some important theories and lemmas that are used extensively in this paper.

First, let us introduce the Littlewood-Paley theory; see \cite{B-C-D-2011} for the details.

\begin{Proposition}[Dyadic Partition of Unity]
Let $\mathcal{C}$ be  annulus $\{\xi \in \R^{d} \,:\, \frac{3}{4}\leq |\xi| \leq \frac{8}{3}\}$.
There exist radial functions $\chi$ and $\va$, valued in interval $[0, 1]$,
belonging to $\mathcal{D}(B(0, \frac{4}{3}))$ and $\mathcal{D}(\mathcal{C})$, respectively,
such that
\begin{align}
& \chi(\xi)+\sum_{j\geq 1}\va (2^{-j}\xi)=1   \qquad\mbox{for any $\xi \in \R^d$}, \notag\\
& \sum_{j\in \mathbb{Z}}\va(2^{-j}\xi)=1 \qquad\mbox{for any $\xi \in \R^d \setminus \{0\}$},\notag\\
&|j-j^{'}|\geq 2\,\, \Longrightarrow \,\,\supp (\va(2^{-j}\cdot))\cap \supp (\va(2^{-j^{'}}\cdot))=\emptyset,\notag\\
&j\geq 1 \,\, \Longrightarrow \,\, \supp (\chi) \cap \supp (\va(2^{-j}\cdot))=\emptyset. \notag
\end{align}
\end{Proposition}

\begin{Definition} 
The homogeneous dyadic blocks $\D_{j}$ and the homogeneous low-frequency cut-off operators $S_{j}$ are defined for all $j\in \mathbb{Z}$ by
\begin{align}
\D_{j}u=\mathcal{F}^{-1}(\va(2^{-j}\xi)\mathcal{F}u)=2^{jd}\int_{\R^d}h(2^{j}y)u(x-y)dy,\notag\\
S_{j}u=\mathcal{F}^{-1}(\chi(2^{-j}\xi)\mathcal{F}u)=2^{jd}\int_{\R^d}\tilde{h}(2^{j}y)u(x-y)dy,\notag
\end{align}
where $\mathcal{F}$ denotes the Fourier transform on $\R^d$, and
\begin{align*}
h:=\mathcal{F}^{-1}\va, \quad \tilde{h}:=\mathcal{F}^{-1}\chi. \notag
\end{align*}

We define the Sobolev norm of space $H^{s}$ as
$$
\|u\|_{H^{s}}:=\big(\|S_{0}u\|_{L^2}^2 +\sum_{j\in \mathbb{N}}2^{2qs}\|\D_{j}u\|_{L^2}^2\big)^{\frac{1}{2}}.
$$

Next, we recall the Bony's paraproduct decomposition for two appropriately smooth functions $u$ and $v$:
\begin{align}
uv=T_{u}v+T_{v}u+R(u, v)=T_{u}v+T^{'}_{v}u, \notag
\end{align}
where
\begin{align}
&T_{u}v:=\sum_{j}S_{j-1}u\D_{j}v, \quad
R(u, v):=\sum_{|j-j^{'}|\leq 1}\D_{j}u\D_{j^{'}}v,\notag\\
&T^{'}_{v}u=T_{v}u+R(u, v)=\sum_{j}S_{j+2}v\D_{j}u.\notag
\end{align}
Then we have
\begin{equation*}
\begin{split}
\D_{j}(uv)&=\D_{j}T_{u}v+\D_{j}T^{'}_{v}u\\
&=\D_{j}\sum_{j^{'}}S_{j^{'}-1}u\D_{j^{'}}v+\D_{j}\sum_{j^{'}}S_{j^{'}+2}v\D_{j^{'}}u\\
&=\sum_{|j-j^{'}|\leq 5}\D_{j}(S_{j^{'}-1}u\D_{j^{'}}v)+\sum_{j^{'}>j-5} \D_{j}(S_{j^{'}+2}v\D_{j^{'}}u)\\
&=S_{j-1}u\D_{j}v+\sum_{|j-j^{'}|\leq 5}[\D_{j}, S_{j^{'}-1}u]\D_{j^{'}}v\\
&\quad +\sum_{|j-j^{'}|\leq 5}(S_{j^{'}-1}u-S_{j-1}u)\D_{j}\D_{j^{'}}v+\sum_{j^{'}>j-5}\D_{j}(S_{j^{'}+2}v\D_{j^{'}}u).
\end{split}
\end{equation*}
\end{Definition}

In order to obtain the uniqueness and higher regularity of the weak-strong solutions in \S 4--\S 5,
we now recall the following useful inequalities that we use extensively.
These inequalities follows from the Bernstein-type lemma and the commutator estimates
in Chapter 2 of \cite{B-C-D-2011} by the construction of $\D_{j}$.
\medskip

\noindent \textbf{Bernstein-type inequalities}:
\begin{align}
&\|\na S_{j}u\|_{L^p}\leq C2^{j}\|u\|_{L^p} \qquad \mbox{for any $1\leq p \leq \infty$}, \notag \\
&C2^{j}\|\D_{j}u\|_{L^p}\leq \|\na \D_{j}u\|_{L^p}\leq C2^{j}\|\D_{j}u\|_{L^p} \qquad \mbox{for any $1\leq p \leq \infty$}, \notag\\
&\|\D_{j}u\|_{L^q}\leq C2^{d(\frac{1}{p}-\frac{1}{q})j}\|\D_{j}u\|_{L^p} \qquad \mbox{with}\,\,\,\, 1\leq p\leq q,\notag\\
&\|S_{j}u\|_{L^q}\leq C2^{d(\frac{1}{p}-\frac{1}{q})j}\|S_{j}u\|_{L^p} \qquad \mbox{with}\,\,\,\, 1\leq p\leq q. \notag
\end{align}
\noindent \textbf{Commutator estimates}:
\begin{align}
\|[\D_{j}, u]v\|_{L^r}\leq C 2^{-j}\|\na u\|_{L^p}\|v\|_{L^q} \qquad \mbox{with}\,\,\,\, \frac{1}{p}+\frac{1}{q}=\frac{1}{r},\notag
\end{align}
where $C$ is independent of $p, q$, and $r$.

In order to deal with the highest derivatives of $u$ in the first equation of system (\ref{Qtensor}) (and the corresponding one
in the approximation systems) and the highest derivatives of $Q$ in the second equation of the system
(and the corresponding one in the approximation systems), we need to introduce the following lemma:

\begin{Lemma}\label{estimate-matrix}
Let $Q$ and $Q^{'}$ be two $d\times d$ symmetric matrices,
and let $\Omega =\frac{1}{2}(\na u-\na u^{\top})$ be the vorticity with $(\na u)_{\a \b}=\del_{\b}u_{\a}$.
Then
\begin{equation*}
(\O Q^{'}-Q^{'}\O, \D Q)-(\na \cdot (Q^{'}\D Q-\D QQ^{'}), u)=0.
\end{equation*}
\end{Lemma}

\begin{proof}
For any two $d\times d$ matrices $A$ and $B$, we know that $\tr (AB)=\tr (BA)$.
Then we have
\begin{align*}
&(\O Q^{'}-Q^{'}\O, \D Q)-(\na \cdot (Q^{'}\D Q-\D QQ^{'}), u)\\
&=(\O Q^{'}-Q^{'}\O, \D Q)+(Q^{'}\D Q-\D QQ^{'}, \na u^{\top})\\
&=(\O Q^{'}, \D Q)-(Q^{'}\O, \D Q)+(Q^{'}\D Q-\D QQ^{'}, \na u^{\top})\\
&=(Q^{'}\D Q, \O)-(\D QQ^{'}, \O)+(Q^{'}\D Q-\D QQ^{'}, \na u^{\top})\\
&=(Q^{'}\D Q-\D QQ^{'}, \O+\na u^{\top})\\
&=(Q^{'}\D Q-\D QQ^{'}, D)=0,
\end{align*}
where, in the last equality, we use the fact that $Q$, $Q^{'}$, and $D$ are symmetric.
\end{proof}

\begin{Remark}
For any $d\times d$ matrix $Q$, we have
\be\label{trace-Q-3}
\tr(Q^3)\leq \frac{\ve}{4}|\tr(Q^2)|^2+\frac{1}{\ve}\tr(Q^2) \qquad \mbox{for any $\ve>0$}.
\en
\end{Remark}

\begin{proof}
Let $x, y$, and $z$ be the eigenvalues of $Q$. Then we have
\begin{equation*}
\tr (Q)=x+y+z, \quad
\tr (Q^2) =x^2 +y^2 +z^2, \quad
\tr (Q^3) =x^3 +y^3 +z^3.
\end{equation*}
Since it is obvious that $\tr(Q^4)\leq |\tr (Q^2)|^2$, we only need to show that, for any $\ve>0$,
the following inequality is true:
\begin{equation*}
\tr(Q^3)\leq \frac{3\ve}{8}\tr(Q^4)+\frac{1}{\ve}\tr(Q^2).
\end{equation*}
By Young's inequality with $\ve$, we have
\begin{align*}
x^3 \leq \frac{x^2}{\ve}+\frac{\ve}{4}x^4,\quad
y^3 \leq \frac{y^2}{\ve}+\frac{\ve}{4}y^4,\quad
z^3 \leq \frac{z^2}{\ve}+\frac{\ve}{4}z^4.
\end{align*}
Then we obtain the desired inequality.
\end{proof}
\begin{Lemma}[Aubin-Lions lemma]\label{al-lemma}
Let $X_{0}, X$, and $X_{1}$ be three Banach spaces with $X_{0}\subseteq X \subseteq X_{1}$,
let $X_{0}$ be compactly embedded in $X$, and let $X$ be continuously embedded in $X_{1}$.
For $1\leq p, q \leq \infty$, let
\begin{equation*}
W=\big\{ u\in L^p (0, T; X_{0}) \, :\, \dot{u}\in L^q (0, T; X_{1})\big\}.
\end{equation*}
Then
\begin{enumerate}
\item[(i)] If $p<\infty$, then the embedding of $W$ into $L^p (0, T; X_{1})$ is compact.\\
\item[(ii)] If $p=\infty$ and $q>1$, then the embedding of $W$ into $C(0, T; X)$ is compact.
\end{enumerate}
\end{Lemma}

\begin{Lemma}[Gagliardo-Nirenberg interpolation inequality \cite{N-1955}]\label{gn-inequality}
Let $1\leq q, r \leq \infty$. For $0\leq j<m$, the following inequalities hold{\rm :}
$$
\|{D}^{j}u\|_{L^p}\leq C \|{D}^{m}u\|_{L^r}^{a}\|u\|_{L^q}^{1-a},
$$
where
$$
\frac{1}{p}=\frac{j}{n}+a(\frac{1}{r}-\frac{m}{d})+(1-a)\frac{1}{q},
$$
for all $a$ in the interval
$$
\frac{j}{m}\leq a \leq 1,
$$
and $C=(d, m, j, q, r, a)$, with the following exceptional cases{\rm :}
\begin{enumerate}
\item[(i)] If $j=0, rm<d$, and $q=\infty$, then we make the additional assumption that either $u$ tends to zero at infinity or $u\in L^{s}$ for some finite $s>0$.

\smallskip
\item[(ii)] If $1<r<\infty$, and $m-j-\frac{d}{r}$ is a nonnegative integer, then it is necessary to assume additionally that $a\neq 1$.
\end{enumerate}
\end{Lemma}

\begin{Lemma}[Gronwall's inequality]\label{Gronwall}
Let $\alpha\ge 0$ and $\beta\ge 0$ be integrable functions on $[0, T]$.
If a differentiable function $Y$ satisfies the differential inequality:
$$
Y'(t) \le \alpha(t)Y(t)+ \beta(t)  \qquad\mbox{for $\, t\in [0,T]$},
$$
then
$$
Y(t)\le Y(0)\exp\Big(\int_0^t\alpha(s)ds\Big)
 +\int_0^t\beta(s)\exp\Big(\int_s^t\alpha(\tau)d\tau\Big)\, ds
\qquad\mbox{for any $\, t\in [0,T]$}.
$$
\end{Lemma}

\bigskip
\section{The Estimates of  Inequality (\ref{high-freq-Q-u})}
Before proving inequality \eqref{high-freq-Q-u}, let us first introduce the following Lemma.

\begin{Lemma}\label{2d-high-power-Q-control}
Let $u\in H^{s}\cap L^{p}$ with $p\geq 1$ and $s>0$.
Then, for any $k\geq 2$ and $q\in \mathbb{N}$,
\be
\|\D_{q}u^{k}\|_{L^{p}}\leq C2^{-qs}a_{q, k}(t) \|u\|_{L^{p(k-1)}}^{k-1}\|\na u\|_{H^{s}},
\en
where $\{a_{q, k}(t)\}_{q\in \mathbb{N}}$ is a sequence in $l^{2}$.
\end{Lemma}

\begin{proof}
We prove this lemma by induction.

Firstly, for $k=2$, by using the Bony's paraproduct decomposition,
we have
\begin{align*}
\D_{q} (u^{2})&=S_{q-1}u\D_{q}u+\sum_{|q-q^{'}|\leq 5}[\D_{q}, S_{q^{'}-1}u]\D_{q^{'}}u\\
&\quad +\sum_{|q-q^{'}|\leq 5}(S_{q^{'}-1}u-S_{q-1}u)\D_{q}\D_{q^{'}}u
+\sum_{q{'}>q-5}\D_{q}(S_{q^{'}+2}u\D_{q^{'}}u)\\
&=\sum_{1\leq i\leq 4}\i_{i}.
\end{align*}
Let us calculate the right side term by term as follows,
\begin{align*}
\|\i_{1}\|_{L^{p}}&=\|S_{q-1}u\D_{q}u\|_{L^{p}}
\leq C\|u\|_{L^{p}}\|\D_{q}u\|_{L^{\infty}}\leq C\|u\|_{L^{p}}\|\D_{q}\na u\|_{L^{2}}\\
&\leq C2^{-qs}\bar{a}^{(1)}_{q}(t) \|u\|_{L^{p}}\|\na u\|_{H^{s}},
\end{align*}
where $\{\bar{a}_{q}^{(1)}(t)\}_{q\in \mathbb{N}}$ is a sequence in $l^{2}$,
\begin{align*}
\|\i_{2}\|_{L^{p}}&=\|\sum_{|q-q^{'}|\leq 5}[\D_{q}, S_{q^{'}-1}u]\D_{q^{'}}u\|_{L^{p}}\leq C\sum_{|q-q^{'}|
\leq 5}2^{-q}\|\na S_{q^{'}-1}u\|_{L^{p}}\|\D_{q^{'}}u\|_{L^{\infty}}\\
&\leq C \sum_{|q-q^{'}|\leq 5}2^{q^{'}-q}\|u\|_{L^{p}}\|\D_{q^{'}}\na u\|_{L^{2}}\\
&\leq C 2^{-qs}\sum_{|q-q^{'}|\leq 5}2^{(q^{'}-q)(1-s)}\bar{a}_{q^{'}}^{(1)}(t)\|u\|_{L^{p}}\|\na u\|_{H^{s}}\\
&\leq C2^{-qs}\bar{a}_{q}^{(2)}(t)\|u\|_{L^{p}}\|\na u\|_{H^{s}},
\end{align*}
where $\{\bar{a}^{(2)}_{q}(t)\}_{q\in \mathbb{N}}=\{\sum_{|q-q^{'}|\leq 5}2^{(q^{'}-q)(1-s)}\bar{a}^{(1)}_{q^{'}}(t)\}_{q\in \mathbb{N}}$ is a sequence in $l^{2}$,
\begin{align*}
\|\i_{3}\|_{L^{p}}&=\|\sum_{|q-q^{'}|\leq 5}(S_{q^{'}-1}u-S_{q-1}u)\D_{q}\D_{q^{'}}u\|_{L^{p}}\\
&\leq C\sum_{|q-q^{'}|\leq 5}\|S_{q^{'}-1}u-S_{q-1}u\|_{L^{p}}\|\D_{q}\D_{q^{'}}u\|_{L^{\infty}}\\
%
%
&\leq C\sum_{|q-q^{'}|\leq 5}\|u\|_{L^{p}}\|\D_{q^{'}}\na u\|_{L^{2}}\\
&\leq C2^{-qs}\sum_{|q-q^{'}|\leq 5}2^{(q-q^{'})s}\bar{a}^{(1)}_{q^{'}}(t)\|u\|_{L^{p}}\|\na u\|_{H^{s}}\\
&\leq C2^{-qs}\bar{a}^{(3)}_{q}(t)\|u\|_{L^{p}}\|\na u\|_{H^{s}},
\end{align*}
with $\{\bar{a}^{(3)}_{q}(t)\}_{q\in \mathbb{N}}=\{\sum_{|q-q^{'}|\leq 5}2^{(q-q^{'})s}\bar{a}^{(1)}_{q^{'}}(t)\}_{q\in \mathbb{N}}\in l^{2}$,
\begin{align*}
\|\i_{4}\|_{L^{p}}&=\|\sum_{q{'}>q-5}\D_{q}(S_{q^{'}+2}u\D_{q^{'}}u)\|_{L^{p}}
\leq C\sum_{q{'}>q-5}\|S_{q^{'}+2}u\|_{L^{p}}\|\D_{q^{'}}u\|_{L^{\infty}}\\
&\leq C\sum_{q{'}>q-5}\|u\|_{L^{p}}\|\D_{q^{'}}\na u\|_{L^{2}}\leq C2^{-qs}\sum_{q{'}>q-5}2^{(q-q^{'})s}\|u\|_{L^{p}}\bar{a}^{(1)}_{q^{'}}(t)\|\na u\|_{H^{s}}\\
&\leq C2^{-qs}\bar{a}^{(4)}_{q}(t)\|u\|_{L^{p}}\|\na u\|_{H^{s}},
\end{align*}
with $\{\bar{a}^{(4)}_{q}(t)\}_{q\in \mathbb{N}}=\{\sum_{q{'}>q-5}2^{(q-q^{'})s}\bar{a}^{(1)}_{q^{'}}(t)\}_{q\in \mathbb{N}}\in l^{2}$.

Then taking $a_{q, 2}(t)=\mathrm{max}_{1\leq k\leq 4}\{\bar{a}_{q}^{(k)}(t)\}$, we know that the result holds for $k=2$.

Next, we first assume the statement is valid for $k$ and then check the validity
for the case $k+1$.
Similarly to the previous steps, we have
\begin{align*}
\D_{q} (u^{k+1})&=\D_{q}(u^{k}u)=S_{q-1}u^{k}\D_{q}u+\sum_{|q-q^{'}|\leq 5}[\D_{q}, S_{q^{'}-1}u^{k}]\D_{q^{'}}u\\
&\quad +\sum_{|q-q^{'}|\leq 5}(S_{q^{'}-1}u^{k}-S_{q-1}u^{k})\D_{q}\D_{q^{'}}u
+\sum_{q{'}>q-5}\D_{q}(S_{q^{'}+2}u\D_{q^{'}}u^{k})\\
&=\sum_{1\leq i\leq 4}\j_{i},
\end{align*}
and
\begin{align*}
\|\j_{1}\|_{L^{p}}&=\|S_{q-1}u^{k}\D_{q}u\|_{L^{p}}
\leq C\|u^{k}\|_{L^{p}}\|\D_{q}u\|_{L^{\infty}}\leq C\|u\|_{L^{pk}}^{k}\|\D_{q}\na u\|_{L^{2}}\\
&\leq C2^{-qs}\bar{a}^{(1)}_{q}(t) \|u\|_{L^{pk}}^{k}\|\na u\|_{H^{s}},
\end{align*}
\begin{align*}
\|\j_{2}\|_{L^{p}}&=\|\sum_{|q-q^{'}|\leq 5}[\D_{q}, S_{q^{'}-1}u^{k}]\D_{q^{'}}u\|_{L^{p}}
\leq C\sum_{|q-q^{'}|\leq 5}2^{-q}\|\na S_{q^{'}-1}u^{k}\|_{L^{p}}\|\D_{q^{'}}u\|_{L^{\infty}}\\
&\leq C \sum_{|q-q^{'}|\leq 5}2^{q^{'}-q}\|u^{k}\|_{L^{p}}\|\D_{q^{'}}\na u\|_{L^{2}}\\
&\leq C 2^{-qs}\sum_{|q-q^{'}|\leq 5}2^{(q^{'}-q)(1-s)}\bar{a}^{(1)}_{q^{'}}(t)\|u\|_{L^{pk}}^{k}\|\na u\|_{H^{s}}\\
&\leq C2^{-qs}\bar{a}^{(2)}_{q}(t)\|u\|_{L^{pk}}^{k}\|\na u\|_{H^{s}},
\end{align*}
\begin{align*}
\|\j_{3}\|_{L^{p}}&=\|\sum_{|q-q^{'}|\leq 5}(S_{q^{'}-1}u^{k}-S_{q-1}u^{k})\D_{q}\D_{q^{'}}u\|_{L^{p}}\\
&\leq C\sum_{|q-q^{'}|\leq 5}\|S_{q^{'}-1}u^{k}-S_{q-1}u^{k}\|_{L^{p}}\|\D_{q}u\|_{L^{\infty}}\\
&\leq C\sum_{|q-q^{'}|\leq 5}\|u^{k}\|_{L^{p}}\|\D_{q}\na u\|_{L^{2}}\\
&\leq C2^{-qs}\bar{a}^{(1)}_{q}(t)\|u\|_{L^{pk}}^{k}\|\na u\|_{H^{s}},
\end{align*}
and from the induction assumption, one has
\begin{align*}
\|\j_{4}\|_{L^{p}}&=\|\sum_{q{'}>q-5}\D_{q}(S_{q^{'}+2}u\D_{q^{'}}u^{k})\|_{L^{p}}\leq C\sum_{q{'}>q-5}\|u\|_{L^{kp}}\|\D_{q^{'}}u^{k}\|_{L^{\frac{kp}{k-1}}}\\
&\leq C\sum_{q{'}>q-5}\|u\|_{L^{kp}}2^{-q^{'}s}a_{q^{'}, k}(t)\|u\|_{L^{kp}}^{k-1}\|\na u\|_{H^{s}}\\
&\leq C2^{-qs}\sum_{q{'}>q-5}2^{(q-q^{'})s}a_{q^{'}, k}(t)\|u\|_{L^{pk}}^{k}\|\na u\|_{H^{s}}\\
&\leq C2^{-qs}\bar{a}_{q, k+1}(t)\|u\|_{L^{pk}}^{k}\|\na u\|_{H^{s}},
\end{align*}
with $\{\bar{a}_{q, k+1}(t)\}_{q\in \mathbb{N}}=\{\sum_{q{'}>q-5}2^{(q-q^{'})s}a_{q^{'}, k}(t)\}_{q\in \mathbb{N}}\in l^{2}$.

Then taking $a_{q, k+1}(t)=\mathrm{max}\{\bar{a}^{(1)}_{q}(t), \bar{a}_{q}^{(2)}(t), \bar{a}_{q, k+1}(t)\}$, we know that the statement is true for the case $k+1$.
By induction, we complete the proof.
\end{proof}

Then, in order to derive \eqref{high-freq-Q-u}, we estimate the terms on the right-hand side of \eqref{high-frequency} one by one.
For the detailed estimates of the following terms,
we refer the readers to the appendix in \cite{P-Z-2012}:
\begin{align*}
&|\i_1| \leq C2^{-2qs}b_{q}(t)\big(\|u\|_{L^2}^{\frac{1}{2}}\|\na u\|_{L^2}^{\frac{1}{2}}\|\na Q\|_{H^s}^{\frac{1}{2}} \|\D Q\|_{H^s}^{\frac{3}{2}}\\
&\qquad\quad
  +\|\na Q\|_{L^2}^{\frac{1}{2}}\|\D Q\|_{L^2}^{\frac{1}{2}}\|u\|_{H^s}^{\frac{1}{2}}\|\na u\|_{H^{s}}^{\frac{1}{2}}\|\D Q\|_{H^s}\big),\\
&|\i_2|+|\i_3|+|\i_5|+|\i_6|\leq C2^{-2qs}b_{q}(t)\|\na Q\|_{L^2}^{\frac{1}{2}}\|\D Q\|_{L^{2}}^{\frac{1}{2}}\|u\|_{H^s}^{\frac{1}{2}}\|\na u\|_{H^{s}}^{\frac{1}{2}}\|\D Q\|_{H^s},\\
&|\i_4|+|\i_7| \leq C2^{-2qs}b_{q}(t)\|u\|_{L^2}^{\frac{1}{2}}\|\na u\|_{L^{2}}^{\frac{1}{2}}\|\na Q\|_{H^s}^{\frac{1}{2}}\|\D Q\|_{H^s}^{\frac{3}{2}},\\
&|\j_1|\leq C2^{-2qs}b_{q}(t)\| u\|_{L^{2}}^{\frac{1}{2}}\|\na u\|_{L^{2}}^{\frac{1}{2}}\|u\|_{^{H^s}}^{\frac{1}{2}}\|\na u\|_{H^s}^{\frac{3}{2}},\\
&|\j_2|\leq C2^{-2qs}b_{q}(t)\|\na Q\|_{L^2}^{\frac{1}{2}}\|\D Q\|_{L^{2}}^{\frac{1}{2}}\|Q\|_{H^s}^{\frac{1}{2}}\|\na Q\|_{H^{s}}^{\frac{1}{2}}\|\na u\|_{H^s},\\
&|\j_3|+|\j_4|+|\j_6|+|\j_7|\leq C2^{-2qs}b_{q}(t)\|\na Q\|_{L^2}^{\frac{1}{2}}\|\D Q\|_{L^{2}}^{\frac{1}{2}}\|u\|_{H^s}^{\frac{1}{2}}\|\na u\|_{H^{s}}^{\frac{1}{2}}\|\D Q\|_{H^s},\\
&|\j_5|+ |\j_8|\leq C2^{-2qs}b_{q}(t)\|\na Q\|_{L^2}^{\frac{1}{2}}\|\D Q\|_{L^{2}}^{\frac{1}{2}}\|\na Q\|_{H^s}^{\frac{1}{2}}\|\D Q\|_{H^{s}}^{\frac{1}{2}}\|\na u\|_{H^s},
\end{align*}
where $\{b_{q}(t)\}_{q\in \mathbb{N}}$ is a sequence in $l^{1}$.

Next we will give the details of the estimates for the remaining terms:
\begin{align*}
|\i_9|&=|\Gamma a(\D_q Q, \D \D_q Q)|=\Gamma a\|\D_q \na Q\|_{L^2}^2 \leq C2^{-2qs}b^{(1)}_{q}(t)\|\na Q\|_{H^s}^2,\\
|\j_{12}|&=|\kappa (\D_q Q, \D_q \na u)|=|\kappa(\D_q \na Q, \D_q u)| \\
&\leq C\|\D_q \na Q\|_{L^2}\|\D_q u\|_{L^2}\leq C2^{-2qs}b^{(2)}_{q}(t)\|\na Q\|_{H^s}\|u\|_{H^s},
\end{align*}
for some $\{b^{(i)}_{q}(t)\}_{q\in \mathbb{N}}\in l^{1}$ with $i=1, 2$.
Applying Lemma \ref{2d-high-power-Q-control}, we have the following estimates:
\begin{align*}
|\i_{10}|&=|\Gamma c(\D_{q}(Q\tr (Q^2)), \D \D_q Q)|\leq C\|\D_{q}(Q\tr (Q^2))\|_{L^{2}}\|\D \D_q Q\|_{L^{2}}\\
&\leq C2^{-qs}a_{q, 3}(t)\|Q\|_{L^{4}}^{2}\|\na Q\|_{H^{s}}2^{-qs}\bar{a}^{(1)}_{q}(t)\|\D Q\|_{H^{s}}\\
&\leq C2^{-2qs}b^{(3)}_{q}(t)\|Q\|_{L^{4}}^{2}\|\na Q\|_{H^{s}}\|\D Q\|_{H^{s}},\\
|\j_{10}|&=|-a\la \big(\D_q (|Q|Q), \na \D_q u\big)|\\
&\leq C\|\D_q (|Q|Q)\|_{L^2}\|\na \D_q u\|_{L^2}\\
&\leq C2^{-qs}a_{q, 2}(t)\|Q\|_{L^{2}}\|\na Q\|_{H^{s}}2^{-qs}\bar{a}^{(1)}_{q}(t)\|\na u\|_{H^s}\\
&\leq C(s)2^{-2qs}b^{(4)}_{q}(t)\|Q\|_{L^2}\|\na Q\|_{H^{s}}\|\na u\|_{H^s},\\
|\j_{11}|&=|-c\la (\D_q (|Q|Q\tr(Q)^2), \na \D_q u)|\\
&\leq C\|\D_q (|Q|Q\tr(Q^2))\|_{L^2}\|\na \D_q u\|_{L^2}\\
&\leq C 2^{-qs}a_{q, 4}(t)\|Q\|_{L^{6}}^{3}\|\na Q\|_{H^{s}}2^{-qs}\bar{a}_{q}^{(1)}(t)\|\na u\|_{H^{s}}\\
&\leq C2^{-2qs}b^{(5)}_{q}(t)\|Q\|_{L^{6}}^{3}\|\na Q\|_{H^{s}}\|\na u\|_{H^{s}},
\end{align*}
for some $\{b^{(i)}_{q}(t)\}_{q\in \mathbb{N}}\in l^{1}$ with $i=3, 4, 5$.

Finally, for the remaining  two terms $\i_{8}$ and $\j_9$ it is not easy to obtain the bounds directly.
The key idea is to combine them together and identify the cancellation factors between them
for the estimates. Applying the Bony's decomposition, we have
\begin{align*}
\i_8&=-\la (\D_q (|Q|D), \D \D_q Q)\\
&=-\la (S_{q-1}(|Q|)\D_q D, \D \D_q Q)-\la \sum_{|q^{'}-q|\leq 5}([\D_q ; S_{q^{'}}(|Q|)]\D_{q^{'}}D, \D \D_{q}Q)\\
&\quad -\la \sum_{|q^{'}-q|\leq 5}((S_{q^{'}-1}|Q|-S_{q-1}|Q|)\D_q \D_{q^{'}}D, \D \D_{q}Q)\\
&\quad -\la \sum_{q^{'}>q-5}(\D_{q}(S_{q^{'}+2}D\D_{q^{'}}|Q|), \D \D_{q}Q)\\
&=\la \sum_{i=1}^{4}\i_{8, i},\\
\j_9&=\la (\D_q (|Q|\D Q), \na \D_{q}u)=\la (S_{q-1}|Q|\D_q \D Q, \na \D_{q}u)+\la ([\D_q ; S_{q^{'}}|Q|]\D_{q^{'}}\D Q, \na \D_{q}u)\\
&\quad +\la \sum_{|q^{'}-q|\leq 5}((S_{q^{'}-1}|Q|-S_{q-1}|Q|)\D_q \D_{q^{'}}\D Q, \na \D_{q}u)\\
&\quad +\la \sum_{q^{'}>q-5}(\D_{q}(S_{q^{'}+2}\D Q\D_{q^{'}}|Q|), \na \D_{q}u)\\
&=\la \sum_{i=1}^{4}\i_{9, i}.
\end{align*}
Since $Q$ is symmetric and $\Omega=\na u -D$ is skew-symmetric, we find that, for $1\leq i \leq 3$,
\begin{equation*}
\i_{8, i}+\j_{9, i}=0.
\end{equation*}
Then
\begin{equation*}
\i_{8}+\j_{9}=\i_{8, 4}+\j_{9, 4}.
\end{equation*}
In addition, we have
\begin{align*}
|\i_{8, 4}|&=\big|\sum_{q^{'}>q-5}(\D_{q}(S_{q^{'}+2}D\D_{q^{'}}|Q|), \D \D_{q}Q)\big|\\
&\leq C\sum_{q^{'}>q-5}\|S_{q^{'}+2}\na u\|_{L^4}\|\D_{q^{'}}|Q|\|_{L^4}\|\D \D_{q}Q\|_{L^2}\\
&\leq C\sum_{q^{'}>q-5}2^{q^{'}+2}\| u\|_{L^4}2^{-q^{'}}\|\D_{q^{'}}\na Q\|_{L^4}\|\D \D_{q}Q\|_{L^2}\\
&\leq C\sum_{q^{'}>q-5}\|u\|_{L^2}^{\frac{1}{2}}\|\na u\|_{L^{2}}^{\frac{1}{2}}
  \|\D_{q^{'}}\na Q\|_{L^2}^{\frac{1}{2}}\|\D_{q^{'}}\D Q\|_{L^2}^{\frac{1}{2}}\|\D_{q}\D Q\|_{H^s}\\
&\leq C2^{-2qs}\sum_{q^{'}>q-5}2^{(q-q^{'})s}\bar{a}^{(1)}_{q^{'}}(t)\bar{a}^{(1)}_{q}(t)
  \|u\|_{L^2}^{\frac{1}{2}}\|\na u\|_{L^{2}}^{\frac{1}{2}}\|\na Q\|_{H^s}^{\frac{1}{2}}\|\D Q\|_{H^s}^{\frac{3}{2}}\\
&\leq C2^{-2qs}b^{(6)}_{q}(t)\|u\|_{L^2}^{\frac{1}{2}}\|\na u\|_{L^{2}}^{\frac{1}{2}}
  \|\na Q\|_{H^s}^{\frac{1}{2}}\|\D Q\|_{H^s}^{\frac{3}{2}},
\end{align*}
and
\begin{align*}
|\i_{9, 4}|&=\big|\sum_{q^{'}>q-5}(\D_{q}(S_{q^{'}+2}\D Q\D_{q^{'}}|Q|), \na \D_{q}u)\big|\\
&\leq C\sum_{q^{'} >q-5}\|S_{q^{'} +2}\D Q\|_{L^4}\|\D_{q^{'}} Q\|_{L^4}\|\D_q \na u\|_{L^2}\\
&\leq C\sum_{q^{'} >q-5}2^{q^{'}}\|\na Q\|_{L^4}2^{-q^{'}}\|\D_{q^{'}} \na Q\|_{L^4}\|\D_q \na u\|_{L^2}\\
&\leq C\sum_{q^{'} >q-5}\|\na Q\|_{L^2}^{\frac{1}{2}}\|\D Q\|_{L^{2}}^{\frac{1}{2}}
  \|\D_{q^{'}} \na Q\|_{L^2}^{\frac{1}{2}}\|\D_{q^{'}}\D Q\|_{L^{2}}^{\frac{1}{2}}\|\D_q \na u\|_{L^2}\\
&\leq C\sum_{q^{'} >q-5}\|\na Q\|_{L^2}^{\frac{1}{2}}\|\D Q\|_{L^{2}}^{\frac{1}{2}}2^{-q^{'}s}\bar{a}^{(1)}_{q^{'}}(t)
  \|\na Q\|_{H^s}^{\frac{1}{2}}\|\D_{q^{'}}\D Q\|_{H^{s}}^{\frac{1}{2}}2^{-qs}\bar{a}^{(1)}_{q}(t)\|\na u\|_{H^s}\\
&\leq C2^{-2qs}b^{(6)}_{q}(t)\|\na Q\|_{L^2}^{\frac{1}{2}}\|\D Q\|_{L^{2}}^{\frac{1}{2}}
  \|\na Q\|_{H^s}^{\frac{1}{2}}\|\D Q\|_{H^{s}}^{\frac{1}{2}}\|\na u\|_{H^s},
\end{align*}
with $\{b^{(6)}_{q}(t)\}_{q\in \mathbb{N}}=\{\sum_{q^{'}>q-5}2^{(q-q^{'})s}\bar{a}^{(1)}_{q^{'}}(t)\bar{a}^{(1)}_{q}(t)\}_{q\in \mathbb{N}}\in l^{1}$.

Multiplying all the above estimates by $2^{2qs}$, adding them together,  taking the sum in $q$ for $q\in \mathbb{N}$,
noticing that $\{b_{q}(t)\}_{q\in \mathbb{N}}, \{b^{(i)}_{q}(t)\}_{q \in\mathbb{N}}\in l^{1}$ with $1\leq i\leq 6$, and using the Cauchy inequality with suitable $\ve$, {\it i.e.}, $ab\leq \ve a^2 +\frac{C}{\ve} b^2$,
we obtain the desired estimate (\ref{high-freq-Q-u}).

\bigskip

\section*{Acknowledgments}
G.-Q. Chen's research was supported in part by
the UK
Engineering and Physical Sciences Research Council Award
EP/L015811/1
and the Royal Society--Wolfson Research Merit Award (UK).
A.  Majumdar's research is supported by an EPSRC Career Acceleration Fellowship EP/J001686/1 and EP/J001686/2,
an OCIAM Visiting Fellowship and the Advanced Studies Centre at Keble College.
D. Wang's research was supported in part by the National Science Foundation under grants DMS-1312800 and DMS-1613213.
R. Zhang's research was supported in part by the National Science Foundation under grant DMS-1312800.

\end{document}